\documentclass[onepage,12pt]{article}

\usepackage[utf8]{inputenc}
\usepackage{amsmath}
\usepackage{amsthm}
\usepackage{amssymb}
\usepackage{times}
\usepackage{bm}
\usepackage[authoryear]{natbib}
\usepackage{amsfonts}
\usepackage{amssymb}
\usepackage{authblk}
\usepackage{mathtools}
\usepackage[margin=0.75in]{geometry}
\usepackage{xcolor}
\usepackage{array}
\usepackage[linktoc=page,linkcolor=blue,colorlinks=true,citecolor=blue,breaklinks=True]{hyperref}
\usepackage{booktabs}
\usepackage{caption}
\usepackage{tikz}
\usepackage{lmodern}
\usepackage{tocloft}
\usepackage{xpatch}
\usepackage{apalike}

\usetikzlibrary{positioning}

\usepackage{graphicx}
\DeclarePairedDelimiterX{\norm}[1]{\lVert}{\rVert}{#1}
\DeclarePairedDelimiterX{\abs}[1]{\lvert}{\rvert}{#1}
\DeclareMathOperator*{\argmax}{arg\,max}
\DeclareMathOperator*{\argmin}{arg\,min}
\newcommand\restr[2]{{% we make the whole thing an ordinary symbol
  \left.\kern-\nulldelimiterspace % automatically resize the bar with \right
  #1 % the function
  \right|_{#2} % this is the delimiter
  }}

\DeclarePairedDelimiterX{\infdivx}[2]{(}{)}{%
	#1\;\delimsize\|\;#2%
}

\newcommand{\KL}{\text{KL}\infdivx}
\DeclareMathOperator*{\esssup}{ess\,sup}
\newcommand{\diff}{\mathrm{d}}
\newcommand{\bbP}{\mathbb{P}}

\newcommand{\bbR}{\mathbb{R}}

\newcommand{\bbE}{\mathbb{E}}
\newcommand{\bbN}{\mathbb{N}}

\newcommand{\calA}{\mathcal{A}}
\newcommand{\calH}{\mathcal{H}}

\newcommand{\calB}{\mathcal{B}}
\newcommand{\calX}{\mathcal{X}}

\newcommand{\calP}{\mathcal{P}}
\newcommand{\calL}{\mathcal{L}}

\newtheorem{problem}{Problem}
\newtheorem{theorem}{Theorem}
\newtheorem{lemma}{Lemma}

\newtheorem{corollary}{Corollary}

\newtheorem{example}{Example}

\title{Bayes Hilbert Spaces for Posterior Approximation}
\author{George Wynne}
\affil{University of Bristol}

\date{}

\begin{document}

\makeatletter
% Patch case where name and year are separated by aysep

\let\Title\@title
% Patch case where name and year are separated by aysep
\patchcmd{\NAT@citex}
{\@citea\NAT@hyper@{%
		\NAT@nmfmt{\NAT@nm}%
		\hyper@natlinkbreak{\NAT@aysep\NAT@spacechar}{\@citeb\@extra@b@citeb}%
		\NAT@date}}
{\@citea\NAT@nmfmt{\NAT@nm}%
	\NAT@aysep\NAT@spacechar\NAT@hyper@{\NAT@date}}{}{}

% Patch case where name and year are separated by opening bracket
\patchcmd{\NAT@citex}
{\@citea\NAT@hyper@{%
		\NAT@nmfmt{\NAT@nm}%
		\hyper@natlinkbreak{\NAT@spacechar\NAT@@open\if*#1*\else#1\NAT@spacechar\fi}%
		{\@citeb\@extra@b@citeb}%
		\NAT@date}}
{\@citea\NAT@nmfmt{\NAT@nm}%
	\NAT@spacechar\NAT@@open\if*#1*\else#1\NAT@spacechar\fi\NAT@hyper@{\NAT@date}}
{}{}
\makeatother
\maketitle
\begin{abstract}
	\noindent {Performing inference in Bayesian models requires sampling algorithms to draw samples from the posterior. This becomes prohibitively expensive as the size of data sets increase. Constructing approximations to the posterior which are cheap to evaluate is a popular approach to circumvent this issue. This begs the question of what is an appropriate space to perform approximation of Bayesian posterior measures. This manuscript studies the application of Bayes Hilbert spaces to the posterior approximation problem. Bayes Hilbert spaces are studied in functional data analysis in the context where observed functions are probability density functions and their application to computational Bayesian problems is in its infancy. This manuscript shall outline Bayes Hilbert spaces and their connection to Bayesian computation, in particular novel connections between Bayes Hilbert spaces, Bayesian coreset algorithms and kernel-based distances. }
\end{abstract}

\tableofcontents
\newpage
%%%%%%%%%%%%%%%%%%%%%%%%%%%%%%%%%%%%%%%%%%%%%%%%%%%%%%%%%%%%%%%%%%%%%%%%%%%%%%%%%%%%%%%%

%!TEX root = ./main.tex

\section{Introduction}
The aim of this manuscript is to advocate for, and demonstrate the utility of, Bayes Hilbert spaces as tools to investigate computational Bayesian problems. 
Bayesian statistics is a common statistical modelling paradigm which involves the fusion of prior knowledge about the unknown quantity to be inferred with observed data. The result is a posterior measure over the unknown quantity. The posterior measure is the object of interest since it can be used to perform prediction and uncertainty quantification. Practically though, using the posterior measure for these tasks involves deploying sampling algorithms on the posterior measure. The practical issue is that it is expensive to sample from the posterior measure when the observed data set is large. The typical cost when one has observed $N$ data points is $O(N)$ per iteration of the sampling algorithm, which becomes prohibitive since many thousands of iterations could be required to obtain the desired number of samples. For more on Bayesian approaches to modelling see \citet{Gelman2013} and for sampling in Bayesian methods see the handbook \citet{Kroese2011}. 

There are a range of solutions to this problem. The two main approaches are designing algorithms which avoid the $O(N)$ cost each iteration and forming an approximation to the target posterior and sampling from that approximation instead of the target posterior. Examples of the former approach involve sampling algorithms which employ sub-sampling \citep{Bierkins2019,Dang2019}, variational inference \citep{Blei2017,Hoffman2013} and divide and conquer methods \citep{Scott2016,Srivastava2015,Vyner2023}. The latter approximation approach will be the focus of this paper. This approach typically involves keeping the prior the same and approximating the likelihood function with something cheaper to evaluate than $O(N)$, resulting in a cheaper per-iteration cost when sampling. The area of this methodology which shall be the focus in the sequel is Bayesian coresets \citep{Huggins2016,Campbell2019}. This is where the an approximation of the likelihood is formed based on $M$ points and weights, with $M\ll N$. This approximation results in an $O(M)$ per-iteration cost once inserted into a sampling algorithm. Many variants of this method exist \citep{Campbell2018,Campbell2019,Campbell2019VI,Manousakas2020,Naik2022}. 

When performing this approximation, indeed when performing any posterior approximation, one must ask themselves what is an appropriate space to form the approximation? This choice plays a critical role in the analysis and performance of posterior approximations. It will impact both the notion of distance between the approximate posterior and the target posterior and how one performs optimisation with respect to this distance to form the approximation. The aim of this paper is to show Bayes Hilbert spaces are appropriate spaces to perform such analysis and facilitate a novel perspectives on Bayesian coreset algorithms and posterior approximation in general. 

Bayes Hilbert spaces \citep{Boogaart2010,Boogaart2014,Barfoot2023} are spaces of measures which have notions of addition and scalar multiplication which are coherent with Bayes's theorem. This makes them particularly appealing spaces to form posterior approximations. The notion of a Bayes Hilbert space has origins in compositional data analysis \citep{Aitchson1982,Glahn2011} and is now a common tool used in distributional data analysis \citep{Mateu2021,Petersen2022}. This is a sub-field of functional data analysis where a user observes a data set which is a collection of probability density functions and wishes to perform analysis. To deal with the specific structure of probability density functions, different versions of addition and scalar multiplication are used in the Bayes Hilbert space than typical function spaces. The use of Bayes Hilbert spaces in distributional data analysis is mature with many applications \citep{Mateu2021}. On the other hand, the use of Bayes Hilbert spaces in computational Bayesian problems is in its infancy. The primary reference of this application is the innovative contribution of \citet{Barfoot2023} which made clear how variational inference procedures can be viewed in terms of projections in Bayes Hilbert spaces. This connection was made by showing variational inference methods can be viewed as optimising coefficients of an approximating function in a Bayes Hilbert space. 

The main technical contributions of this manuscript are relating the Bayes Hilbert space norm to commonly used distances on measures, providing a novel connection between Bayes Hilbert spaces and kernel methods and framing multiple Bayesian coreset algorithms in terms of Bayes Hilbert spaces.

Section \ref{sec:Background} will provide background content for this manuscript. In particular, Section \ref{sec:BHS} will introduce the mathematical details of Bayes Hilbert spaces, Section \ref{sec:Bayes_theorem} will describe the sense in which Bayes Hilbert spaces are coherent with Bayes' theorem and Section \ref{sec:existing} will outline existing literature related to Bayes Hilbert spaces. After the background content, Section \ref{sec:bounds} will provide a novel result which relates the distance in a Bayes Hilbert space to common discrepancies for measures, Section \ref{sec:Bayes_Coresets} will outline how Bayesian coresets can be expressed in a Bayes Hilbert space, Section \ref{sec:MMD} will provide a novel link between maximum mean discrepancy and Bayesian coresets, Section \ref{sec:Hilbert_Coresets} will leverage this link to provide novel connections between maximum mean discrepancy algorithms and Hilbert coreset algorithms, Section \ref{sec:KL_coresets} will describe a novel connection between Bayesian coresets constructed using the Kullback-Leibler divergence and Bayes Hilbert spaces. Concluding remarks and future avenues of research are outlined in Section \ref{sec:Conclusion}.

\section{Background}\label{sec:Background}
This section shall provide background information to motivate and frame the novel results that will occur in the later sections. Section \ref{sec:BHS} shall outline Bayes Hilbert spaces which are the central focus of this manuscript, Section \ref{sec:Bayes_theorem} shall outline Bayes' theorem and Section \ref{sec:existing} shall discuss lines of research related to the Bayes Hilbert spaces methodology, with a focus on methods which map measures into Hilbert spaces. 

%!TEX root = ./main.tex

\subsection{Bayes Hilbert Spaces}\label{sec:BHS}
The aim of this section is to motivate and introduce the mathematical details of Bayes Hilbert spaces. This will include defining the elements of a Bayes Hilbert space, the notions of addition and scalar multiplication and the Hilbert inner product. 

Before beginning with the mathematical details of a Bayes Hilbert space it is helpful to understand how the ideas developed. The ideas behind a Bayes Hilbert space may be traced back to compositional data analysis \citep{Aitchson1982}. Mathematically, compositional data is non-negative multivariate data with a sum constraint that represents relative, rather than absolute, information. This data type is appropriate when one wishes to analyse the composition of certain objects, hence the name compositional. For example, one may be studying a collection of chemicals which are each made up from a fixed number of other chemicals. Then compositional data analysis can be used to perform analysis on the relative proportions of each of the chemicals. For a long list of other examples see \citet{Aitchison1986}.

This type of data has many quirks. For example, a coherent notion of addition and subtraction is not straightforward. This is because the typical notion of vector addition would lead to a notion of subtraction which could result in negative values, and negative values cannot represent proportions. Issues therefore also occur if one were to try and use the typical notion of scalar multiplication with this type of data. 

The pioneering paper by \citet{Aitchson1982} derived a mathematical framework to deal with the quirks of compositional data outlined above. These tools are known as the Aitchison geometry and form the bedrock for the field of compositional data analysis. The geometry revolves around deriving particular notions of equality, addition, scalar multiplication and subtraction for compositional data along with a logarithmic transform which maps the data from the simplex into a nicer space. The application areas of this mathematical framework are legion, covering areas such as geology, ecology and microbiome research \citep{Glahn2011,Filzmoser2018,Aitchison1986,vandenBoogaart2013} as well as an insightful survey paper marking 40 years since Aitchison's original publication \citep{Greenacre2022}.

The idea of a Bayes Hilbert space is to generalise the Aitchison geometry used in compositional data analysis to the situation where the objects of interest are measures. The easiest way to conceptualise this move is to view compositional data as a histogram, with each the $n$-th entry of a compositional vector representing the probability of the $n$-th event. Then, the trick is to view a probability density function as a continuous limit of a histogram and to then adapt the Aitchison geometry to this continuous limit, which in practice means moving from finite sums over the entries of a compositional vector to integrals with respect to probability density functions. Then the move from probability density functions to measures is made by using Radon-Nikodym derivatives. 

The rest of this section will be spent making these ideas and intuition concrete. All technical content is taken from existing sources \citep{Hron2022,Boogaart2014,Boogaart2010,Maier2021}.

Let $(\Theta,\calA)$ be a measurable space and $\mu$ a probability measure on $\Theta$. This measure $\mu$ will act as the base measure. What follows can be generalised to finite measures that are not probability measures, the difference is simply additional normalisation terms. Elements of $\Theta$ will be denoted by $\theta$. Define $M(\mu)$ as the set of measures on $(\Theta,\calA)$ that are $\sigma$-finite and mutually absolutely continuous with respect to $\mu$, which means that they have the same null sets as $\mu$.  Discussion regarding the choice of $\mu$ for Bayesian computation problems is given in Section \ref{sec:Bayes_Coresets}. 

Define an equivalence relation $=_{B}$ between two measure $\eta,\nu\in M(\mu)$ as $\eta =_{B}\nu$ if and only if there exists a constant $c > 0$ such that $\eta(A)=c\nu(A)\:\forall \:A\in\calA$. It can be easier to understand this notion of equivalence in terms of the Radon-Nikodym derivatives with respect to $\mu$. For a measure $\eta\in M(\mu)$ the Radon-Nikodym derivative of $\eta$ with respect to $\mu$ is the non-negative function $\frac{\diff\eta}{\diff\mu}$ such that $\eta(A) = \int_{A}\frac{\diff\eta}{\diff\mu}\diff\mu\:\forall A\in\calA$. The existence of such a function is guaranteed by the Radon-Nikodym theorem, see for example \citet[Theorem 3.2.2]{Bogachev2007}. The function is unique $\mu$-almost everywhere. The interpretation of this result is simply that measures in $M(\mu)$ possess what can be viewed as densities with respect to $\mu$ which makes interpretation of results to come more straightforward. To this end, for $\eta\in M(\mu)$ set $p_{\eta,\mu} \coloneqq \frac{\diff\eta}{\diff\mu}$ as the Radon-Nikodym derivative of $\eta$ with respect to $\mu$ for notational convenience. Note that $p_{\mu,\mu} = 1$ in the sense that $p_{\mu,\mu}$ is the function constantly equal to one. Then $\eta=_{B}\nu$ if and only if there exists a constant $c>0$ such that $p_{\eta,\mu} = cp_{\nu,\mu}$ $\mu$-almost everywhere. Indeed, anytime that these Radon-Nikodym derivatives are written as equal it is to be interpreted in the $\mu$-almost everywhere sense. 

Define $B(\mu)$ to be the set of equivalence classes with respect to $=_{B}$ within $M(\mu)$. With this notion of equality in place, which is helpful when studying Bayesian problems as described in Section \ref{sec:Bayes_Coresets}, the next step is to define addition and scalar multiplication. For $\eta,\nu\in B(\mu)$ addition is defined
\begin{align*}
	(\eta\oplus\nu)(A) \coloneqq \int_{A}p_{\eta,\mu}(\theta)\cdot p_{\nu,\mu}(\theta)\diff\mu(\theta)
\end{align*}
for every $A\in\calA$, For $\eta,\nu\in B(\mu)$ and $\alpha\in\bbR$ scalar multiplication is defined as
\begin{align*}
	(\alpha\odot\eta)(A) \coloneqq \int_{A}p_{\eta,\mu}(\theta)^{\alpha}d\mu(\theta)
\end{align*}
for every $A\in\calA$. The combination of addition and scalar multiplication facilitates the identification of an additive inverse, namely for $\nu\in B(\mu)$ define $\ominus \nu\coloneqq (-1\odot\nu)$ and for $\eta,\nu \in B(\mu)$ define $\eta\ominus\nu \coloneqq \eta\oplus(\ominus \nu)$. 

While these definitions may look abnormal at first, all of the operations $\oplus,\odot,\ominus$ can be interpreted straightforwardly through the Radon-Nikodym derivatives. Namely, for any $\eta,\nu\in B(\mu)$ the operation $\oplus$ can be written as $p_{\eta,\mu}\oplus p_{\nu,\mu}\coloneqq p_{\eta,\mu}\cdot p_{\nu,\mu}$ where $\cdot$ means standard multiplication so that $p_{\eta\oplus\nu,\mu} = p_{\eta,\mu}\oplus p_{\nu,\mu}$. Scalar multiplication $\odot$ can be written as $\alpha \odot p_{\eta,\mu}\coloneqq p_{\eta,\mu}^{\alpha}$ so that $p_{\alpha\odot\eta,\mu} =\alpha \odot p_{\eta,\mu}$. Finally, $\ominus$ can be written as $p_{\eta\ominus\nu,\mu} = p_{\eta,\mu}/p_{\nu,\mu}$ so that $p_{\eta\ominus\nu,\mu} = p_{\eta,\mu} - p_{\eta,\mu}$. These operations defined directly on the Radon-Nikodym derivatives make the natural identification $\eta$ with $ p_{\eta,\mu}$ for all $\eta\in B(\mu)$ coherent in the sense that if one writes $\eta =_{B} p_{\eta,\mu}\:\forall\eta\in B(\mu)$ then $\eta\oplus\nu =_{B} p_{\eta\oplus\nu,\mu}$ for all $\eta,\nu\in B(\mu$), similarly for $\odot,\ominus$. Therefore, for ease of notation $\eta$ may be replaced with $p_{\eta,\mu}$ at certain places in the sequel as at times the focus shall be on the Radon-Nikodym derivatives rather than the measures themselves.

The next result links all these operations together and assures us that these operations result in a valid real vector space.

\begin{theorem}\citep[Theorem 5]{Boogaart2010}\label{prop:vector_space}
	$B(\mu)$ equipped with $\oplus,\odot$ is a real vector space with $\mu$ the additive zero element. 
\end{theorem}

This vector space is known as a \textit{Bayes linear space} \citep{Boogaart2010} and provides the basic structure needed to proceed to Bayes Hilbert spaces. The next step is defining the following set
\begin{align}
	B^{2}(\mu) = \left\{\eta\in B(\mu)\colon \bbE_{\mu}\left[\left\lvert\log \frac{\diff\eta}{\diff\mu}\right\rvert^{2}\right] < \infty\right\}.\label{eq:BHS_set}
\end{align}
Note that $\eta\in B^{2}(\mu)$ if and only if $\log p_{\eta,\mu}\in L^{2}(\mu)$, where $L^{2}(\mu)$ is the typical space of (equivalence classes of) functions on $\Theta$ that are square integrable with respect to $\mu$. 

As previously mentioned, the central tool for compositional data analysis was a logarithm transform to map the compositional data into a nice space. The analogy to this for the present case is the \textit{centred log-ratio} (CLR) transform which is defined \citep{Egozcue2006,Boogaart2014}
\begin{align*}
	\Psi_{\mu}(\eta) = \log p_{\eta,\mu} - \bbE_{\mu}[\log p_{\eta,\mu}],
\end{align*}
where the expectation is being taken with respect to the input argument of $\log p_{\eta,\mu}$. A few remarks are in order. First, the expectation is finite given the assumption that $\mu$ is a finite measure and that $\eta\in B^{2}(\mu)$. Second, since $p_{\eta,\mu}$ accepts as argument an element of $\Theta$ so to does the CLR $\Psi_{\mu}(\eta)$. Third, $\Psi_{\mu}(\mu) = 0$, in the sense that it is the function constantly zero, since $p_{\mu,\mu} = 1$ meaning that the base measure, which is the additive identity in $B^{2}(\mu)$ is mapped to the zero function.  The CLR maps into $L^{2}_{0}(\mu) = \{f\in L^{2}(\mu)\colon \bbE_{\mu}[f] = 0\}$, the elements of $L^{2}(\mu)$ that have zero mean, which is the analogy to the nice space that is used in compositional data analysis. This space is also used in information geometry as a tangent space to manifolds of measures, see Section \ref{sec:existing} for more discussion. Finally, the integrability condition in $B^{2}(\mu)$ is very weak and means that $B^{2}(\mu)$ can contain infinite measures. Discussion on this is given in Section \ref{sec:Conclusion}.

The operations $\oplus,\odot$ relate addition with multiplication and scalar multiplication with exponentiation, these relations are maintained by logarithms which is why the CLR transform is helpful. Also, the subtraction of the expectation of the logarithm ensures the scale invariance property in the definition of $=_{B}$ holds. Overall, this means for any $\eta,\nu\in B^{2}(\mu)$ and $\alpha\in\bbR$ the following linearity holds $\Psi_{\mu}(\alpha\odot(\eta\oplus\nu)) = \alpha\cdot\left(\Psi_{\mu}(\eta)+\Psi_{\mu}(\nu)\right)$ and that for any $\eta,\nu\in B^{2}(\mu)$ such that $\eta =_{B}\nu$ the CLR maps $\eta,\nu$ to the same output, meaning $\Psi_{\mu}(\eta) = \Psi_{\mu}(\nu)$. The inverse function of the CLR map is simply $\exp$, to see this take any $\eta\in B^{2}(\mu)$ then 
\begin{align*}
	\exp(\Psi_{\mu}(\eta)) =_{B} p_{\eta,\mu}\cdot\exp(-\bbE_{\mu}[\log p_{\eta,\mu}]) =_{B}p_{\eta,\mu}
\end{align*}
due to the scale invariance of $=_{B}$ \citep{Boogaart2010}, where the natural identification between $\eta$ and $p_{\eta,\mu}$, as described immediately before Theorem \ref{prop:vector_space}, has been used. 

With the CLR established an inner product structure can be defined. For $\eta,\nu\in B^{2}(\mu)$ define
\begin{align}
	\langle \eta,\nu\rangle_{B^{2}(\mu)} \coloneqq\langle \Psi_{\mu}(\eta),\Psi_{\mu}(\nu)\rangle_{L^{2}(\mu)} = \int_{\Theta}\Psi_{\mu}(\eta)(\theta)\Psi_{\mu}(\nu)(\theta)\diff\mu(\theta).\label{eq:inner_prod}
\end{align}
with corresponding norm
\begin{align}
	\norm{\eta - \nu}_{B^{2}(\mu)} =\norm{\Psi_{\mu}(\eta) - \Psi_{\mu}(\nu)}_{L^{2}(\mu)}.\label{eq:BHS_norm}
\end{align}
The next result shows the inner product on $B^{2}(\mu)$ provides the desired Hilbertian structure. 

\begin{theorem}\citep[Theorem 1]{Boogaart2014}\label{thm:BHS_L2}
	The map $\Psi_{\mu}$ is an isometry between $B^{2}(\mu)$ and $L^{2}_{0}(\mu) = \{f\in L^{2}(\mu)\colon \bbE_{\mu}[f] = 0\}$. The inverse of $\Psi_{\mu}$ is $\exp$ and $B^{2}(\mu)$ is a Hilbert space. 
\end{theorem} 
The space $B^{2}(\mu)$ equipped with the inner product \eqref{eq:inner_prod} is called a  \textit{Bayes Hilbert space}. Theorem \ref{thm:BHS_L2} reveals critical structure of $B^{2}(\mu)$. This Hilbertian structure will facilitate approximations based on dictionaries of functions and optimisation methods based on orthogonal projections. Such structure is typically not present in common representations of measures and therefore Bayes Hilbert spaces have great potential for applications in measure approximation. For example, depending on the choice of $\mu$ it can be straightforward to derive an orthonormal basis for $B^{2}(\mu)$ by using orthogonal polynomials. The case of $\mu$ being a Gaussian and using Hermite polynomials is studied in \citet{Barfoot2023} in the context of Gaussian variational inference. More discussion regarding the potential use of these basis approximations for computational Bayesian problems is given in Section \ref{sec:Conclusion}.

%!TEX root = ./main.tex

\subsection{Bayes' Theorem and Bayes Hilbert Spaces}\label{sec:Bayes_theorem}
The aim of this section is to explain how Bayes Hilbert spaces have structure which is coherent with Bayes' theorem, justifying their name. In particular, it will be shown that the notions of addition and scalar multiplication used in $B^{2}(\mu)$ are coherent with using Baye's theorem to update a prior belief using likelihoods to obtain a posterior. This section contains no novel technical content with exposition taken from \citet{Boogaart2010,Boogaart2014}.

One may go back to compositional data analysis, the genesis of Bayes Hilbert spaces as outlined in Section \ref{sec:BHS}, to see how the connection to Bayes's theorem occurs. Aitchison noted that the simplex, the canonical sample space for compositional data analysis, is ``familiar in other areas of statistics $\ldots$ as the operation of Bayes’s formula to change a
prior probability assessment into a posterior probability assessment through the perturbing influence of the likelihood function'' \citep{Aitchison1986,Boogaart2010}.

These allusions will now be made concrete. The following recites content from \citet{Boogaart2010} and for more on Bayes' theorem consult \citet{Stuart2010}. The sample space will be $\Theta$ and the prior measure $\pi_{0}$. Suppose there are $N$ observations $\{x_{n}\}_{n=1}^{N}\subset\calX$, the data space, that are independently and identically distributed. Assume the likelihood for each observation is $l$ which takes as input an observation from $\calX$ and a parameter from $\Theta$. Then, assuming the observations are independent and identically distributed, the likelihood for all the data is $L(\theta) \coloneqq\prod_{n=1}^{N}l(\theta,x_{n})$, or using shorter notation $L = \prod_{n=1}^{N}l_{x_{n}}$ where $l_{x_{n}}(\theta) \coloneqq l(\theta,x_{n})$. Bayes' theorem states the posterior measure $\pi$ over $\Theta$, given the prior, likelihood and observations, satisfies 
\begin{align}
	\frac{\diff\pi}{\diff\pi_{0}} = Z^{-1}L,\label{eq:Bayes_rule}
\end{align}
where $Z = \bbE_{\pi_{0}}[L]$ is a constant and is known as the evidence. It is implicitly assumed that the prior, likelihood and data are such that $Z<\infty$ which is a very mild and common assumption. 

Mathematically, \eqref{eq:Bayes_rule} shows that the Radon-Nikodym derivative of the posterior with respect to the prior is proportional to the likelihood. Therefore, using the definitions of $=_{B}$ and $p_{\pi,\pi_{0}}$ given in Section \ref{sec:BHS}
\begin{align}
	p_{\pi,\pi_{0}} =_{B}L =_{B}\bigoplus_{n=1}^{N}l_{x_{n}}. \label{eq:post_in_BHS} 
\end{align} 
This means shows that Bayes' theorem is simply a sum over the likelihood terms when written in Bayes Hilbert space notation. This is the sense in which the Bayes Hilbert space operations are coherent with Bayes' theorem. Another helpful property of a Bayes Hilbert space is the definition of $=_{B}$ as being equality up to scalar constant. This is useful since in practice posterior sampling algorithms are agnostic of scalar constants and therefore a mathematical framework for posterior approximation should also be invariant to scalar constants. 

\begin{example}\label{exp:log_reg_intro}
	An example of Bayesian inference that will recur is logistic regression. Suppose one observes input output pairs $x_{n} = \{u_{n},y_{n}\}$ with $u_{n}\in\bbR^{d}$ and $y_{n}\in\{0,1\}$. For example, $u_{n}$ might represent medical information about the $n$-th person in a population and $y_{n}$ indicates whether they do or do not have a certain illness. A model for the regression problem of predicting $y_{n}$ given $u_{n}$ is logistic regression \citep[Chapter 3.7]{Gelman2013}. The parameter space is $\Theta = \bbR^{d}$ and for parameter choice $\theta$ the response $y_{n}$ is modelled as ${y_{n}\sim\emph{Bernoulli}(1/1+e^{-\langle \theta,u_{n}\rangle_{\bbR^{d}}})}$ meaning the likelihood is
	\begin{align*}
		l_{x_{n}}(\theta) = \frac{y_{n}}{1+e^{-\langle \theta,u_{n}\rangle_{\bbR^{d}}}} + \frac{(1-y_{n})e^{-\langle \theta,u_{n}\rangle_{\bbR^{d}}}}{1+e^{-\langle \theta,u_{n}\rangle_{\bbR^{d}}}}.
	\end{align*}
	A commonly used prior for $\theta$ is $\pi_{0} = \emph{N}(0,I_{d})$, the standard multivariate Gaussian. 
\end{example}

Under the assumption that $\pi\in B^{2}(\mu)$ the CLR for $\pi$ is 
\begin{align*}
	\Psi_{\mu}(\pi) = \log p_{\pi,\mu} - \bbE_{\mu}[\log p_{\pi,\mu}].
\end{align*}
Under the additional assumption $\pi_{0}\in B^{2}(\mu)$, using the chain rule for Radon-Nikodym derivatives $p_{\pi,\mu} = p_{\pi,\pi_{0}}p_{\pi_{0},\mu}$ the CLR transform can be written 
\begin{align}
	\Psi_{\mu}(\pi) & = \log p_{\pi,\pi_{0}} - \bbE_{\mu}[\log p_{\pi,\pi_{0}}] + (\log p_{\pi_{0},\mu} - \bbE_{\mu}[\log p_{\pi_{0},\mu}])\nonumber\\
	& = \calL - \bbE_{\mu}\left[\calL\right] + (\log p_{\pi_{0},\mu} - \bbE_{\mu}[\log p_{\pi_{0},\mu}]) \label{eq:CLR_post},
\end{align}
where $\calL\coloneqq \log L = \sum_{n=1}^{N}\log l_{x_{n}}$ is the log-likelihood. Note that if $\mu = \pi_{0}$ then the final term in brackets is zero since $p_{\pi_{0},\pi_{0}} = 1$ and the CLR transform becomes equal to the centred log-likelihood function. 

\begin{example}
	Taking $\mu = \pi_{0}$ then continuing Example \ref{exp:log_reg_intro} the CLR for the logistic regression model is
	\begin{align*}
		\Psi_{\pi_{0}}(\pi) & = \sum_{n=1}^{N}\log\left( \frac{y_{n}}{1+e^{-\langle \theta,u_{n}\rangle_{\bbR^{d}}}} + \frac{(1-y_{n})e^{-\langle \theta,u_{n}\rangle_{\bbR^{d}}}}{1+e^{-\langle \theta,u_{n}\rangle_{\bbR^{d}}}}\right)\\
		& - \sum_{n=1}^{N}\bbE_{\pi_{0}}\left[ \log\left(\frac{y_{n}}{1+e^{-\langle \theta,u_{n}\rangle_{\bbR^{d}}}} + \frac{(1-y_{n})e^{-\langle \theta,u_{n}\rangle_{\bbR^{d}}}}{1+e^{-\langle \theta,u_{n}\rangle_{\bbR^{d}}}}\right)\right].
	\end{align*}

\end{example}

The importance of these representations from a computational Bayesian point of view should be considered. The main focus in computational Bayesian statistics is the posterior measure. A computational Bayesian typically wishes to draw samples from the posterior and will almost always have to use a sampling algorithm. The typical cost is $O(N)$ per iteration of the sampling algorithm, with many thousands of iterations typically required, therefore the overall cost becomes prohibitive for large data sets. In this large data scenario one usually employs either a sampling algorithm which makes cheap approximations, such as subsampling data during the execution of the sampling algorithm, or will try and form a cheap approximation of the posterior and sample from that instead. The difficulty in the latter approach is that measures are not typically viewed as lying in a space that is applicable to standard approximation and optimisation theory. 

The Bayes Hilbert space provides appealing structure in which to view the posterior. First, it possesses an appropriate notion of equality as it is agnostic to scaling constants, much like common sampling algorithms. Second, the Hilbertian structure facilitates the use of approximation theory, concentration inequalities and projections which will all play a role in the rest of this manuscript to construct approximations to the posterior. Later, Theorem \ref{thm:bounds} in Section \ref{sec:bounds} shows that the norm in a Bayes Hilbert space can bound commonly used distances between measures and therefore approximation in a Bayes Hilbert space results in minimising commonly used distances.

\subsection{Related Work}\label{sec:existing}
The central idea of Bayes Hilbert spaces is to construct a Hilbert space into which measures shall be mapped. The Hilbert space provides nice structure to form distances, construct approximations and perform optimisation. The application of Bayes Hilbert spaces to computational Bayesian problems beyond expressing Bayes rule in terms of Bayes Hilbert space operations \citep{Boogaart2010} is currently modest. This is because the application of Bayes Hilbert spaces has so far focused on distributional data analysis \citep{Petersen2022,Mateu2021}. The sole contribution to the problem of posterior approximation using a Bayes Hilbert space is \citet{Barfoot2023} which frames variational inference in terms of a Bayes Hilbert space.

Other frameworks exists which embed measures into other function spaces. In information geometry, the most similar topic are exponential manifolds \citep{Pistone1995,Cena2006,Pistone2013}. Informally, this is where a manifold is constructed over the space of probability measures such that at a given measure, call it the reference measure, a Banach space is constructed and a subset found which is homeomorphic to the space of measures that are absolutely continuous with respect to the reference measure. The connection between this method and Bayes Hilbert spaces is that $B^{2}(\mu)$ plays the roles of the reference measure space and $L^{2}_{0}(\mu)$ the Banach space, in this context known as the tangent space. In fact, the CLR transform also features in this method, see \citet[Equation 23]{Pistone1995}. In \citet{Pistone1995} the Banach space was an Orlicz space which enforces stronger integrability conditions than $L^{2}_{0}(\mu)$ to ensure that the corresponding measures are all finite, unlike in $B^{2}(\mu)$. Connections between information geometry and compositional data analysis have been made by \citet{Erb2021} but extending these to Bayes Hilbert spaces was left for future work. 

Orlicz spaces are rather technical and other work has focused on constructing similar manifolds which use Hilbert spaces instead. This includes \citet{Newton2012} which also uses $L^{2}_{0}(\mu)$ as its Hilbert space but uses a map slightly different from the CLR transform, see \citet[Equation 6]{Newton2012}. A reproducing kernel Hilbert space approach, see Section \ref{sec:MMD}, was performed by \citet{Fukumizu2009} in which the Banach space is a reproducing kernel Hilbert space. The CLR transform appears again in \citet[Theorem 18.1]{Fukumizu2009}. This approach is appealing since reproducing kernel Hilbert spaces are very nice spaces to study and manipulate. However, one main difference is that the set of measures absolutely continuous with respect to a given base measure is only homeomorphic to a subset of a Hilbert space. This means it does not inherit the vector space structure from the Hilbert space. This is in contrast to how Bayes Hilbert spaces have vector space structure. 

The purpose of all these methods in information geometry, as the name of the field suggests, is to better understand the geometry of spaces of measures rather than to directly approximate them. The closest to tackling an approximation problem is \citet{Sriperumbudur2011inf} which uses the reproducing kernel Hilbert space approach of \citet{Fukumizu2009} to construct density estimators from samples of a distribution, which is still a distinct task to posterior approximation. Another distinguishing factor between the Bayes Hilbert space approach and the exponential manifold approach is that in the latter the space of measures is bijective with a subset of a function space, not a subspace like the former, since they are tangent spaces to the manifold. More discussion is given in Section \ref{sec:Conclusion} regarding the potential application of exponential manifolds to the problems discussed in this manuscript.  

Another area which uses maps from measures into Hilbert spaces are kernel mean embeddings \citep{Muandet2017,Gretton2012,Berlinet2004,Guilbart1978}. This method is discussed at length in Section \ref{sec:MMD}. The idea is to represent a measure using a kernel and then reason with the representation of the measure rather than the measure itself, for example forming distances and performing hypothesis testing. This method has also been applied to Bayesian computation \citep{Fukumizu2013}. The crucial difference between this approach and the Bayes Hilbert space and exponential manifold approach is that the map from measures to the Hilbert space is not easy to invert, both theoretically and practically. This makes the kernel mean embedding approach difficult to apply to the posterior approximation. The kernel mean embedding distance will appear in Theorem \ref{thm:MMD_BHS} as a way to relate distances in the Bayes Hilbert space to distances between the empirical measures of observed data in posterior approximation. 

%!TEX root = ./main.tex

\section{Measure Discrepancies and Bayes Hilbert Spaces}\label{sec:bounds}
The aim of this section is to prove a novel result which shows the Bayes Hilbert space norm between two measures upper bounds typical notions of discrepancy between measures. Specifically, Theorem \ref{thm:bounds} provides bounds for the Hellinger, Kullback-Leibler and Wasserstein-$1$ distances in terms of a Bayes Hilbert space norm. Such bounds are important if the Bayes Hilbert space is to be used for posterior approximation as they show that the approximation is meaningful. The bounds are obtained by a simple application of a recent result from the Bayesian inverse problems literature regarding the stability of posterior measures \citep{Sprungk2020}. This result is possible since the notion of error of posterior approximation employed within Bayesian inverse problem literature coincides with the distance in Bayes Hilbert spaces. 

Before the result is stated some notation needs to be outlined. First,  $\Theta$ is assumed to be a separable, complete metric space with metric $d_{\Theta}$ and let $\calB(\Theta)$ denote the set of finite Borel measures on $\Theta$ and $\calP(\Theta)\subset\calB(\Theta)$ the subset of probability measures. In all three following definitions, $\eta,\nu\in\calB(\Theta)$.

Assuming $\eta,\nu$ are both absolutely continuous with respect to some measure $\lambda$, such a $\lambda$ always exists for example $\lambda = \eta + \nu$, then the square of the Hellinger distance is defined
\begin{align}
	\text{H}(\eta,\nu)^{2} = \frac{1}{2}\int_{\Theta}\left(\sqrt{\frac{\diff\eta}{\diff\lambda}} -\sqrt{\frac{\diff\nu}{\diff\lambda}} \right)^{2}\diff\lambda,\label{eq:Hellinger}
\end{align} 
where this definition is invariant to the choice of $\lambda$. 

Assuming $\eta$ is absolutely continuous with respect to $\nu$ the Kullback-Leibler (KL) divergence is defined
\begin{align*}
	\KL{\eta}{\nu} = \int_{\Theta}\log\frac{\diff\eta}{\diff\nu}\diff\eta,
\end{align*}
for more discussion on the KL divergence see Section \ref{sec:KL_coresets}.

Let $\text{Lip}(1)$ denote the set of functions $f\colon\Theta\rightarrow\bbR$ with Lipschitz constant less than or equal to one, then the Wasserstein-$1$ distance is defined
\begin{align*}
	\text{W}_{1}(\eta,\nu) = \sup_{f\in\text{Lip}(1)}\left\lvert\int_{\Theta}f\diff\eta - \int_{\Theta}f\diff\nu\right\rvert.
\end{align*}
For $\eta\in\calP(\Theta)$ define 
\begin{align*}
	\norm{\eta}_{\calP^{2}} & = \inf_{\theta_{0}\in\Theta}\left(\int_{\Theta}d_{\Theta}(\theta,\theta_{0})^{2}\diff\eta(\theta)\right)^{1/2}\\
	\calP^{2}(\Theta) & = \{\eta\in\calP(\Theta)\colon\norm{\eta}_{\calB^{2}}<\infty\},
\end{align*} 
which is a quantity which will occur in the Wasserstein-$1$ bound that provides a notion of size of the sample space with respect to the measure and the metric on the space. Finally, for $f\in L^{2}(\mu)$, $\esssup_{\mu}f$ is the essential supremum of $f$ and if there is an $\tilde{f}$ in the $L^{2}(\mu)$-equivalence class of $f$ such that $\tilde{f}(\theta)\leq B\:\forall\theta\in\Theta$ then $\esssup_{\mu}f\leq B$. For a more detailed description of essential supremum consult \citet[Section 2.11]{Bogachev2007}.

\begin{theorem}\label{thm:bounds}
	Let $\Theta$ be a complete, separable metric space and $\mu\in\calP^{2}(\Theta),\eta,\nu\in\calB(\Theta)$. Let $\eta,\nu\in B^{2}(\mu)$ with
	\begin{align}
		\frac{\diff\eta}{\diff\mu} & = Z_{\eta}^{-1}\exp(\Psi_{\mu}(\eta))\qquad Z_{\eta} = \bbE_{\mu}[\exp(\Psi_{\mu}(\eta))]\label{eq:post_RN}\\
		\frac{\diff\nu}{\diff\mu} & = Z_{\nu}^{-1}\exp(\Psi_{\mu}(\nu))\qquad Z_{\nu} = \bbE_{\mu}[\exp(\Psi_{\mu}(\nu))].\label{eq:apprx_RN}
	\end{align}
	Assume that $\esssup_{\mu}\Psi_{\mu}(\eta),\esssup_{\mu}\Psi_{\mu}(\nu)\leq B$ then 
	\begin{align*}
		\emph{H}(\eta,\nu) & \leq \frac{1}{2}\left(e^{B}+e^{2B}\right)^{1/2}\norm{\eta-\nu}_{B^{2}(\mu)} \\
		\emph{KL}\infdivx{\eta}{\nu} & \leq 2e^{B}\norm{\eta-\nu}_{B^{2}(\mu)}\\
		\emph{W}_{1}(\eta,\nu) & \leq (e^{B} + e^{2B})\norm{\mu}_{\calP^{2}}\norm{\eta-\nu}_{B^{2}(\mu)}.
	\end{align*}
\end{theorem}
The proof is contained in the Appendix and is a straightforward application of the results in \citet{Sprungk2020}. A bound using an $L^{1}(\mu)$ norm also holds for the KL and Wasserstein-$1$ case and $\mu$ being a probability measure can be relaxed to $\mu$ being a finite measure as the cost of an extra constant, as discussed in the proof. 
So long as the upper bound $B$ is controlled, Theorem \ref{thm:bounds} shows that the Bayes Hilbert space norm converging to zero implies that the other three notions of distance converge to zero too. The most similar result is \citet[Proposition 6.3]{Huggins2018} which provides a bound on the Wasserstein-$1$ and Wasserstein-$2$ distances, the difference is that an $L^{2}$ norm is used which involves derivatives of the log-likelihoods rather than the Bayes Hilbert space norm. Corollary \ref{cor:bounds} in Section \ref{sec:Bayes_Coresets} deals with the specific case of approximating a posterior measure. 

Implicit in all of this is the choice of the base measure $\mu$. The choice does not impact the left hand side of the bounds and will impact the right hand side by changing $B$ and the value of the norm. Theorem \ref{thm:bounds} provides a guiding principal of choosing $\mu$ to make $B$ as small as possible.  

\begin{example}\label{exp:log_reg_B}
	Continuing the logistic regression example, since $l_{x_{n}}\leq 1$ the log-likelihood is bounded as $\log l_{x_{n}}\leq 0$ therefore the CLR is bounded as
	\begin{align*}
					\Psi_{\pi_{0}}(\pi)\leq - \sum_{n=1}^{N}\bbE_{\pi_{0}}\left[ \log\left(\frac{y_{n}}{1+e^{-\langle \theta,u_{n}\rangle_{\bbR^{d}}}} + \frac{(1-y_{n})e^{-\langle \theta,u_{n}\rangle_{\bbR^{d}}}}{1+e^{-\langle \theta,u_{n}\rangle_{\bbR^{d}}}}\right)\right],
	\end{align*}
	which can be used as a bound in Theorem \ref{thm:bounds}.
\end{example}

The notation for the $Z_{\eta},Z_{\nu}$ terms in the Radon-Nikodym derivatives in \eqref{eq:post_RN},\eqref{eq:apprx_RN} is atypical since they do not factor out all the scalar factors. Specifically, since the CLR map $\Psi_{\mu}$ involves centering by a constant, see \eqref{eq:CLR_post}, there is still a multiplicative factor which could be absorbed into the normalising terms. This was done when introducing the posterior in \eqref{eq:Bayes_rule} with the $Z$ term. The reason this is not done in Theorem \ref{thm:bounds} is to retain the full expressions for Bayes Hilbert space norm in the right hand side.

\begin{example}
	Continuing the logistic regression example, the $Z$ term in \eqref{eq:Bayes_rule} would be 
	\begin{align*}
		Z = \bbE_{\pi_{0}}\left[L\right] = \bbE_{\pi_{0}}\left[\exp\left(\calL\right)\right]
	\end{align*} 
where $L = \prod_{n=1}^{N}l_{x_{n}},\calL = \log L$. In contrast to this the $Z_{\pi}$ term in Theorem \ref{thm:bounds} is 
\begin{align*}
	Z_{\pi} = \bbE_{\pi_{0}}\left[\exp(\Psi_{\pi_{0}}(\pi))\right] = \bbE_{\pi_{0}}\left[\exp(\calL - \bbE_{\pi_{0}}[\calL])\right].
\end{align*}
\end{example}

Typically, bounds of the type used in Theorem \ref{thm:bounds}, see \citet{Sprungk2020}, assume that $\calL\leq 0$ and involve $Z^{-1}$ as a constant, where $Z = \bbE_{\pi_{0}}[L]$. Note $\calL\leq 0$ is a very mild assumption and can be obtained when assuming first $\calL \leq C$ for constant $C$ and then using $\calL - C$, see \citet[Section 2]{Sprungk2020}. The task of upper bounding $Z^{-1}$ is closely related to the upper bound assumption on $\Psi_{\pi_{0}}$ in Theorem \ref{thm:bounds}. To see this note that $Z^{-1} = \bbE_{\pi_{0}}[\exp(\calL)]^{-1}\leq \exp(-\bbE_{\pi_{0}}[\calL])$ by Jensen's inequality. When $\calL\leq 0$ then the bound $B$ only needs to satisfy $-\bbE_{\pi_{0}}[\calL]\leq B$. Hence $Z^{-1}\leq e^{B}$. This shows how despite using different normalising constants in the Radon-Nikodym derivative representation, to ensure the bounds still involve the CLR transform, Theorem \ref{thm:bounds} recovers constants in the upper bounds typical of results in the Bayesian inverse problems literature, for example \citet[Theorem 5]{Sprungk2020}, with the difference being the slack in Jensen's inequality.

%!TEX root = ./main.tex

\section{Bayesian Coresets and Bayes Hilbert Spaces}\label{sec:Bayes_Coresets}
The aim of this section will be to introduce Bayesian coreset posterior approximations \citep{Campbell2019,Huggins2016} and provide a novel relationship between them and Bayes Hilbert spaces. This connection is crucial for the rest of this manuscript and will be analysed further in later sections. 

As mentioned in the previous section, for a likelihood that is a product of $N$ terms, the typical cost per iteration of a sampling algorithm is $O(N)$. This is because typical sampling algorithms have to ``touch'' each data point in the likelihood during each of their iterations. The idea of a Bayesian coreset is to approximate the likelihood with a weighted product of $M$ terms. The resulting approximation will have $O(M)$ cost which can be substantially cheaper if $M\ll N$. The first paper on Bayesian coresets was focused on logistic regression \citep{Huggins2016}. Since then the area has developed widely and has produced multiple innovations, both in terms of how the distance to the full posterior is evaluated and in the optimisation methods used to form the coreset \citep{Campbell2018,Campbell2019,Manousakas2020,Zhang2021,Naik2022}.

Retaining the same setting as Section \ref{sec:Bayes_theorem}, the prior is $\pi_{0}$, the observed data $\{x_{n}\}_{n=1}^{N}\subset\calX$, the individual likelihood functions $l_{x_{n}}(\theta)\coloneqq l(\theta,x_{n})$ and the full likelihood $L = \prod_{n=1}^{N}l_{x_{n}}$. For the rest of this manuscript, a Bayesian coreset is a collection of points $\{z_{m}\}_{m=1}^{M}\subset\calX$ and non-negative weights $\{w_{m}\}_{m=1}^{M}\subset \bbR_{\geq 0}$ for some $M\in\bbN$ and the Bayesian coreset posterior approximation corresponding to the Bayesian coreset is the measure $\pi_{w,z}$ satisfying
\begin{align}
	\frac{\diff \pi_{w,z}}{\diff \pi_{0}} =  Z_{w,z}^{-1} L_{w,z},\label{eq:coreset_aprx} 
\end{align}
where $L_{w,z} = \prod_{m=1}^{M}l_{z_{m}}^{w_{m}}$ and $Z_{w,z} = \bbE_{\pi_{0}}[L_{w,z}]$. Non-negative weights are used so that the resulting posterior is valid in the sense that it has finite measure. This approximation keeps the prior the same and approximates the posterior by using the likelihood $L_{w,z}$ instead of $L$. Note that if $M=N$ and  $w_{m} = 1,z_{m}=w_{m}\:\forall m\:$ then $L = L_{w,z}$ and if $w_{m} = 0$ then $l_{z_{m}}$ plays no role in the approximation. 

The idea behind this approximation is that some of the observed data points may be redundant and do not need to be included, or they may be accounted for by re-weighting a likelihood term based on a different data point. The fact that $\{z_{m}\}_{m=1}^{M}$ does not have to be a subset of the observed data $\{x_{n}\}_{n=1}^{N}$ facilitates a flexible and efficient approximation. The canonical example in this case is a Gaussian mean inference task. In this scenario the posterior can be written exactly with $M=1$, see \citet[Section 3]{Manousakas2020}. The scenario where $\{z_{m}\}_{m=1}^{M}\not\subset\{x_{n}\}_{n=1}^{N}$ was named a \textit{pseudo-coreset} in \citet{Manousakas2020}. This terminology is not maintained in the present discussion since the general case, where no assumptions are placed on $\{z_{m}\}_{m=1}^{M}$ so it could or could not be a subset of $\{x_{n}\}_{n=1}^{N}$, is investigated. 

The CLR transform of $\pi_{w,z}$ is easily obtained by using \eqref{eq:CLR_post}. Let $\mu$ be the base measure and assume that $\pi_{w,z},\pi_{0}\in B^{2}(\mu)$, then
\begin{align}
	\Psi_{\mu}(\pi_{w,z}) = \calL_{w,z} - \bbE_{\mu}\left[\calL_{w,z}\right] + (\log p_{\pi_{0},\mu} - \bbE_{\mu}[\log p_{\pi_{0},\mu}]),\label{eq:CLR_coreset}
\end{align}
where $\calL_{w,z} \coloneqq \log L_{w,z} = \sum_{m=1}^{M}w_{m}\log l_{z_{m}}$ is the weighted log-likelihood based on $\{z_{m}\}_{m=1}^{M}$.

The Bayesian coreset approximation is leveraging the structure inherent in the Bayesian methodology that the Radon-Nikodym derivative of the posterior with respect to the prior is equal, up to a scalar constant, to a product of likelihoods. This means the full posterior can be recovered exactly with the choice $M=N$ and  $w_{m} = 1,z_{m}=w_{m}\:\forall m\:$. Therefore, exact recovery only involves a finite number of parameter choices. This is different to typical function approximation methods where the target function can only be recovered exactly with an infinite number of parameter choices, for example an infinite basis expansion. 

Some simple choices of $\mu$ in the Bayesian coreset scenario are are $\mu=\pi_{0}$, setting $\mu$ to be the posterior based upon a random subset of the data points and a Laplace approximation to the posterior \citep{Campbell2019}. It is currently an open question as to how to chose $\mu$ and what it means for $\mu$ to be a best choice. Intuitively one hopes that $\mu$ has most mass where the posterior has most mass but this causes a chicken and the egg problem as the entire task is to approximate the posterior. as mentioned after Theorem \ref{thm:bounds} a guide can be to choose a $\mu$ which minimises the bound $B$ using in Theorem \ref{thm:bounds}.

The expressions for $\pi,\pi_{w,z}$ can be inserted into Theorem \ref{thm:bounds} to understand how the Bayes Hilbert space norm implies a bound on the common notion of distance for probability measures in this particular case. A typical scenario for Bayesian coresets is having $M=N$ and $z_{n}=x_{n}$ and only changing the weights, with some weights being zero to provide sparsity. The next lemma applies Theorem \ref{thm:bounds} to this case.

\begin{lemma}\label{cor:bounds}
	Under the assumptions of Theorem \ref{thm:bounds}, if $\calL\leq C$ for some constant $C$, $M=N$, $z_{n}=x_{n}\:\forall n$, and the weights $w_{n}$ are non-negative with $\norm{w}_{\infty}\leq W$ then $B$ can be set as $B =-\max(1,W)\bbE_{\mu}[\calL]+C$.
\end{lemma}
\begin{proof}
	Under the assumptions, $\Psi_{\mu}(\pi) = \calL -C - \bbE_{\mu}[\calL]+C\leq -\bbE_{\mu}[\calL]+C$ with the analogous argument for $\Psi_{\mu}(\pi_{w,z})$. Using the assumptions on $w_{m},z_{m}$
	\begin{align*}
		-\bbE_{\mu}[\calL_{w,z}] + C = \sum_{n=1}^{N}-w_{m}\bbE_{\mu}[l_{z_{m}}] + C& \leq \max(1,W)\sum_{n=1}^{N}-\bbE_{\mu}[l_{x_{n}}]+C\\
		& = -\max(1,W)\bbE_{\mu}[\calL] + C,
	\end{align*} 
	which completes the proof. 
\end{proof}

Lemma \ref{cor:bounds} shows that if one forms a Bayesian coreset using points that are a subset of the observed data points, which is very common \citep{Campbell2018,Campbell2019,Naik2022}, then as long as the weights are bounded the term $B$ will be agnostic of the choice of weights. The requirement that the weights are bounded is not typically explicitly enforced in existing Bayesian coreset algorithms although it is reasonable. This is because the norm of the full likelihood is $\norm{\Psi_{\mu}(\pi)}_{L^{2}(\mu)}$ so one would expect each weight to not be much larger than this overall norm value. An exception where the weights are bounded is in the importance sampling coreset methods \citep{Huggins2016}.

Armed with the approximation scheme \eqref{eq:coreset_aprx} and the ability to bound common notions of discrepancy in terms of the Bayes Hilbert space norm using Corollary \ref{cor:bounds} and Theorem \ref{thm:bounds}, the question is now how to choose the weights and the points in the Bayesian coreset approximation. Many methods have emerged in the literature over the past five years or so and the subsequent sections will be dedicated to viewing these through the lens of a Bayes Hilbert space. The first step in this process begins in the next section which makes a novel connection between the Bayes Hilbert space distance and maximum mean discrepancy distance, a kernel-based distance for measures.

\section{Maximum Mean Discrepancy and Bayes Hilbert Spaces}\label{sec:MMD}
This aim of this section is to outline a novel relationship between maximum mean discrepancy (MMD) and Bayesian coresets by relating the MMD distance to a Bayes Hilbert space distance. The relationship is made via the kernel pre-image problem \citep[Chapter 18]{Schlkopf2002}. This relationship will later be used in Section \ref{sec:Hilbert_Coresets} to investigate a genre of Bayesian coreset methods called Hilbert Bayesian coresets. The rest of this section will introduce MMD then the kernel pre-image problem and finally make the relation to Bayesian coresets.

MMD is a kernel-based discrepancy between measures and has been studied for nearly two decades in statistical machine learning \citep{Gretton2006,Gretton2012,Muandet2017} and computational statistics \citep{Teymur2021,Dwivedi2021,Mak2018}, with its origins being in the 70s in more abstract mathematical statistics under a different name \citep{Guilbart1978}. MMD has both theoretical and practical strengths. The theoretical strengths include the elegant representation in terms of a reproducing kernel Hilbert space (RKHS), explained below, which facilitates MMD being rewritten in multiple helpful ways. The main practical strength of MMD is that it can be easily estimated by simply evaluating the kernel on samples from the two measures in question. 

To begin the description of MMD one must first start with a kernel. A kernel $k\colon\calX\times\calX\rightarrow\bbR$ is a symmetric, positive definite function. The point of a kernel is that it measures similarity between its inputs. A simple way to construct a kernel is to take any function $\phi\colon\calX\rightarrow H$, which will be called a feature map, where $H$ is some real Hilbert space, called the feature space, then set $k(x,y) = \langle \phi(x),\phi(y)\rangle_{H}$. In this case the kernel is measuring the similarity between $x$ and $y$ by measuring the similarity of the features $\phi(x)$ and $\phi(y)$ using the inner product in the feature space.  

For any kernel $k$ there is a space of functions called the reproducing kernel Hilbert space that is uniquely determined by the kernel which has specific properties relating to the kernel. In particular, the RKHS, denoted $\calH_{k}$, is the unique Hilbert space of functions from $\calX$ to $\bbR$ such that $k(x,\cdot)\in\calH_{k}\:\forall\:x\in\calX$ and $\langle f,k(x,\cdot)\rangle_{k} = f(x)\:\forall f\in\calH_{k},x\in\calX$ where $\langle\cdot,\cdot\rangle_{k}$ is the inner product of $\calH_{k}$, see \citet[Section 4.2]{Steinwart2008}. The latter property is called the reproducing property and it is extremely useful since, for functions in $\calH_{k}$, it facilitates many quantities of interest to be written purely in terms of the kernel, which is easy to evaluate. For more details regarding kernels and their RKHS consult \citep{Steinwart2008,Berlinet2004,Paulsen2016}.

Before outlining MMD a point on notation. Latin letters will be used to denote measures and random variables on $\calX$, the space where observed data for the Bayesian inference problem will lie, as is common in the MMD literature \citep{Sriperumbudur2010}. This is to distinguish them from measures on $\Theta$, the parameter space, which use Greek letters as was done in Section \ref{sec:BHS}. 

MMD is now defined, for a more in depth discussion than what is presented here consult \citet{Sriperumbudur2010}. Let $\calB_{k}(\calX) = \{P\in \calB(\calX)\colon \bbE_{P}[\sqrt{k(X,X)}]<\infty\}$ where in the expectation $X$ denotes the random variable in $\calX$ with law $P$ and $\calB(\calX)$ is the set of finite Borel measures on $\calX$. Note that all finite sums of empirical measures are in $\calB_{k}$. For $P,Q\in\calB_{k}(\calX)$ the MMD using kernel $k$ is
\begin{align}
	\text{MMD}_{k}(P,Q) \coloneqq \sup_{\norm{f}_{k}\leq 1}\left\lvert\bbE_{P}[f] - \bbE_{Q}[f]\right\rvert,\label{eq:MMD_sup}
\end{align}
where the supremum is being taken over the unit ball of the RKHS. The key property of MMD is that it may be re-written in different forms to make it more interpretable and easily estimated. 

The first of these involves kernel mean embeddings. This is a mapping of a measure into the RKHS via the kernel. For a kernel $k$ the kernel mean embedding $\Phi_{k}\colon\calB_{k}(\calX)\rightarrow\calH_{k}$ is defined 
\begin{align*}
	\Phi_{k}(P) \coloneqq \int_{\calX}k(x,\cdot)\diff P(x).
\end{align*}
The idea of a kernel mean embedding is to provide a feature expansion in the RKHS of a measure $P$, much like how standard kernel methods on $\bbR^{d}$ involve feature expansions of the data. 

Analogous to how the reproducing property represents pointwise evaluation of functions in the RKHS, kernel mean embeddings represent integration of functions in the RKHS. For $f\in\calH_{k}$ and $P\in\calB_{k}(\calX)$ \begin{align}
	\bbE_{P}[f] = \int_{\calX}f(x)\diff P(x) = \int_{\calX}\langle f,k(x,\cdot)\rangle_{k}\diff P(x) = \langle f,\int_{\calX}k(x,\cdot)\diff P(x)\rangle_{k} = \langle f,\Phi_{k}(P)\rangle_{k},\label{eq:int_swap}
\end{align}
where the swapping of the integral and inner product is a result of the integrability assumption on $k$ and $P$ made when assuming $P\in\calB_{k}(\calX)$ \citep[Theorem 1]{Sriperumbudur2010}.

Kernel mean embeddings relate to MMD in the following elegant way
\begin{align}
		\text{MMD}_{k}(P,Q) & = \sup_{\norm{f}_{k}\leq 1}\left\lvert\bbE_{P}[f] - \bbE_{Q}[f]\right\rvert \nonumber\\
		& = \sup_{\norm{f}_{k}\leq 1}\left\lvert\langle f,\Phi_{k}(P)-\Phi_{k}(Q)\rangle_{k}\right\rvert \label{eq:swap_deriv}\\
		& = \norm{\Phi_{k}(P)-\Phi_{k}(Q)}_{k},\label{eq:MMD_Phi}
\end{align}
where \eqref{eq:swap_deriv} is by \eqref{eq:int_swap} and \eqref{eq:MMD_Phi} is by the Cauchy-Schwarz theorem. This shows that MMD is equal the the difference in the RKHS of the feature expansions of the two measures. 

A third representation of MMD can be obtained by starting with \eqref{eq:MMD_Phi} and again using \eqref{eq:int_swap} 
\begin{align}
	\text{MMD}_{k}(P,Q)^{2} & = \norm{\Phi_{k}(P)-\Phi_{k}(Q)}_{k}^{2}\nonumber\\
	& = \langle \Phi_{k}(P),\Phi_{k}(P)\rangle_{k} -2\langle \Phi_{k}(P),\Phi_{k}(Q)\rangle_{k}+ \langle \Phi_{k}(Q),\Phi_{k}(Q)\rangle_{k}\label{eq:rep_expand}\\
	& = \bbE_{P\times P}[k(X,X')] - 2\bbE_{P\times Q}[k(X,Y)] + \bbE_{Q\times Q}[k(Y,Y')],\label{eq:MMD_double_E}
\end{align}
where \eqref{eq:rep_expand} is simply expanding the RKHS norm and \eqref{eq:MMD_double_E} is by using \eqref{eq:int_swap}. In the derivation above $X$ has law $P$, $Y$ has law $Q$ and $X',Y'$ are i.i.d.\ copies of $X,Y$, respectively. This double expectation formula is what makes MMD practical since it can be easily estimated empirically given samples from $P,Q$ \citep{Gretton2012}.

MMD has found many uses in machine learning and computational statistics, for example two-sample testing \citep{Gretton2012}, parameter inference \citep{Cherief2020}, training generative models \citep{Li2017} and distribution compression \citep{Dwivedi2021}. A long list of other applications and references may be found in \citet[Section 3.5]{Muandet2017}. 

The use for MMD that is the focus of this section is the following pre-image problem \citep[Chapter 18]{Schlkopf2002}.

\begin{problem}[The pre-image problem]\label{prob:MMD}
	Given a kernel $k$, $\{x_{n}\}_{n=1}^{N}\subset\calX$ and $M\in\bbN$ find non-negative weights $\{w_{m}\}_{m=1}^{M}\subset\bbR_{\geq 0}$ and  $\{z_{m}\}_{m=1}^{M}\subset\calX$ which minimise
	\begin{align*}
		\emph{MMD}_{k}(P_{w,z},P_{x})^{2} & = \left\lVert\sum_{m=1}^{M}w_{m}k(z_{m},\cdot) - \sum_{n=1}^{N}k(x_{n},\cdot)\right\rVert_{k}^{2} \\
		& = \sum_{m,m'=1}^{M}w_{m}w_{m'}k(z_{m},z_{m'}) - 2\sum_{m=1}^{M}\sum_{n=1}^{N}w_{m}k(z_{m},x_{n}) + \sum_{n,n'=1}^{N}k(x_{n},x_{n'}),
	\end{align*}
	where $P_{w,z} = \sum_{m=1}^{M}w_{m}\delta_{z_{m}}, P_{x} = \sum_{n=1}^{N}\delta_{x_{n}}$ are the empirical measures corresponding to the two sets of data.
\end{problem}

The pre-image problem was studied heavily in the 90s and early 00s in the context of support vector machines \citep{Burges1996,Scholkopf1999RS} and kernel principal component analysis \citep{Kwok2004,Mika1998}. It is called a pre-image problem since it is seeking to find weights and points which map close to $\sum_{n=1}^{N}k(x_{n},\cdot)$ in the RKHS. 

This problem relates to many topics within kernel methods and computational statistics. Reduced set methods solve a similar pre-image problem to find sparse representations of kernel-based algorithms \citep{Burges1996}. Kernel principal component analysis also solves a similar pre-image problem \citep{Schlkopf1997}. Kernel herding \citep{Chen2010,Bach2012} is a method of solving the pre-image problem by greedily picking points and then choosing weights. The main two methods either use uniform weights or line-search. More will be discussed about kernel herding in Section \ref{sec:Hilbert_Coresets}. The scenario where $\{z_{m}\}_{m=1}^{M}\subset\{x_{n}\}_{n=1}^{N}$ is known as distribution compression \citep{Dwivedi2021}, or quantisation \citep{Teymur2021,Graf2000}, where the weights are typically left as uniform over the data points. The scenario where a user is not trying to approximate an empirical distribution but rather a continuous distribution using kernel-based approaches is covered within the kernel Stein discrepancy literature \citep{Riabiz2022,Anastasiou2023} with quasi-Monte Carlo being a related field which aims to find discretisations of given measures for integral approximation \citep{Caflisch1998,Dick2013}. In statistical depth, the $h$-depth \citep{Wynne2021} is an instance of the pre-image problem when $M=1$ so only one point is used to represent the data. 

The pre-image problem immediately appears to be related to Bayesian coresets since both involve starting with a data set and finding weights and points. The subtlety is that the effectiveness of a Bayesian coreset is measured by the quality of its corresponding posterior approximation. This means the notion of quality of weights and points lying in $\calX$ is expressed in terms of a distance between measures on $\Theta$. This is in contrast to most investigations of the pre-image problem and the related problems outlined above. The difference is due to the common use of kernels which involve expressions that measure distance between their inputs purely in terms of the geometry of $\calX$. The trick to link the pre-image problem \eqref{prob:MMD} to Bayesian coresets is to use a kernel which maps the data into a Bayes Hilbert space. This will facilitate comparison of points in $\calX$ in terms of corresponding posteriors.

This is done by using a kernel of the form $k(x,y) = \langle \phi_{\mu}(x),\phi_{\mu}(y)\rangle_{H}$, such kernels were discussed at the start of this section. The same setting and notation for Bayes' theorem used in Section \ref{sec:Bayes_theorem} is maintained. The likelihood function based at a point $x\in\calX$ is $l_{x}$. For some base measure $\mu$ the feature map is $\phi_{\mu}(x)  = \log l_{x} - \bbE_{\mu}[\log l_{x}]$ with feature space $H = L^{2}(\mu)$ and the kernel is defined
\begin{align}
	k(x,y) = \langle \phi_{\mu}(x),\phi_{\mu}(y)\rangle_{L^{2}(\mu)}.\label{eq:BHS_kernel}
\end{align}

This kernel measures the similarity between two points in $\calX$ by comparing the similarity of the centred versions of the log-likelihoods based at those points. This provides the crucial link between the data space and the sample space. 

\begin{theorem}\label{thm:MMD_BHS}
	Let $\calX$ be a metric space equipped with its Borel $\sigma$-algebra. Let $\mu$ be a finite measure on the measurable space $\Theta$ and $l$ a likelihood function such that $\phi_{\mu}\colon\calX\rightarrow L^{2}(\mu)$, $\phi_{\mu}(x) = \log l_{x}-\bbE_{\mu}[\log l_{x}]$ is well-defined and measurable. Let $\pi_{0}\in B^{2}(\mu)$ and $\pi\in B^{2}(\mu)$ be the posterior with prior $\pi_{0}$, likelihood $l$ and observations $\{x_{n}\}_{n=1}^{N}\subset\calX$. Let $M\in\bbN$,  $\{w_{m}\}_{m=1}^{M}\subset\bbR$ and $\{z_{m}\}_{m=1}^{M}\subset\calX$. If $k$ is the kernel \eqref{eq:BHS_kernel} then 
	\begin{align}
		\emph{MMD}_{k}&(P_{w,z},P_{x}) =  \left\lVert\pi_{w,z}-\pi\right\rVert_{B^{2}(\mu)},\label{eq:MMD_thm}
	\end{align}
	where $\pi_{w,z}$ is the posterior based on the weighted likelihood \eqref{eq:coreset_aprx}. 
\end{theorem}
\begin{proof}
The assumptions ensure that $k$ is well-defined and measurable and that empirical measures on $\calX$ are measurable.
	\begin{align}
		\text{MMD}_{k}&(P_{w,z},P_{x})^{2}  = \left\lVert\sum_{m=1}^{M}w_{m}k(z_{m},\cdot) - \sum_{n=1}^{N}k(x_{n},\cdot)\right\rVert_{k}^{2}\label{eq:deriv_3}\\
		& = \sum_{m,m'=1}^{M}w_{m}w_{m'}k(z_{m},z_{m'}) - 2\sum_{m=1}^{M}\sum_{n=1}^{N}w_{m}k(z_{m},x_{n}) + \sum_{n,n'=1}^{N}k(x_{n},x_{n'})\label{eq:deriv_4}\\
		& = \left\lVert\sum_{m=1}^{M}\phi_{\mu}(z_{m}) - \sum_{n=1}^{N}\phi_{\mu}(x_{n})\right\rVert_{L^{2}(\mu)}^{2}\label{eq:deriv_5}\\
		& = \left\lVert\Psi_{\mu}(\pi_{w,z}) - \Psi_{\mu}(\pi)\right\rVert_{L^{2}}^{2}(\mu)\label{eq:deriv_6}\\
		& = \left\lVert\pi_{w,z}-\pi\right\rVert_{B^{2}(\mu)}^{2},\label{eq:deriv_7}
	\end{align}
	where \eqref{eq:deriv_3} is by \eqref{eq:MMD_Phi}, \eqref{eq:deriv_4} is by the reproducing property, \eqref{eq:deriv_5} is by the definition of $k$, \eqref{eq:deriv_6} is by the expression for the CLR transform of the approximation posterior \eqref{eq:CLR_coreset} and the full posterior \eqref{eq:CLR_post} where the term in the brackets in the expressions cancel out and finally \eqref{eq:deriv_7} is by the definition of the Bayes Hilbert space norm \eqref{eq:BHS_norm}.
\end{proof}

Theorem \ref{thm:MMD_BHS} shows that the pre-image problem, Problem \ref{prob:MMD}, is equivalent to minimising a Bayes Hilbert space distance between the Bayesian coreset posterior and the target posterior when using the kernel \eqref{eq:BHS_kernel}. An immediate consequence of this result is that bounds on MMD can be translated into bounds on the Bayes Hilbert space distance, and therefore by Theorem \ref{thm:bounds} bounds on commonly used distances between measures. An simple example is the following corollary which deals with the case of approximating a posterior by using a likelihood based on a subset of the observed data uniformly sampled without replacement. This random uniform subsampling without replacement is used as a common basic benchmark for evaluating approximate posterior methods \citep{Campbell2019}.

\begin{corollary}\label{cor:concentration}
	Under the same assumptions as Theorem \ref{thm:MMD_BHS}, let $M\leq N$ and set $\pi_{M}$ to be the approximate posterior satisfying $p_{\pi_{M},\pi_{0}} =_{B} (N/M)\odot\bigoplus_{m=1}^{M} l_{z_{m}}$ where $\{z_{m}\}_{m=1}^{M}$ is randomly uniformly subsampled without replacement from $\{x_{n}\}_{n=1}^{N}$. Assume that $\norm{\phi_{\mu}(x_{n})}_{L^{2}(\mu)}\leq \gamma\:\forall \:1\leq n\leq N$ then with probability at least $1-\delta$
	\begin{align*}
		\left\lVert\pi-\pi_{M}\right\rVert_{B^{2}(\mu)} \leq \sqrt{\frac{8\gamma^{2}N(N-M+1)}{M}}\sqrt{2\log\left(\frac{2}{\delta}\right)}.
	\end{align*}
\end{corollary}

The proof is a direct application of the kernel mean embeddings concentration inequality of \citet{Schneider2016}. The bound is similar to other concentration inequalities for Bayesian coresets, such as \citet[Theorem 4.1]{Campbell2019}. The difference is that Corollary \ref{cor:concentration} focuses on the simple random sub-sample case. More sophisticated sub-sampling methods, such as importance sampling as was done in \citet[Theorem 4.1]{Campbell2019} could also be easily applied in the context of a Bayes Hilbert space due to the Hilbertian structure.  

The connection between MMD and Bayes Hilbert spaces will be used again in the next section to outline a novel connection between Bayesian coreset algorithms and kernel-based methods for solving the pre-image problem.

\section{Hilbert Bayesian Coresets}\label{sec:Hilbert_Coresets}
The aim of this section is to provide a novel connection between Hilbert Bayesian coreset algorithms and Bayes Hilbert spaces. The former are Bayesian coreset algorithms which revolve around using a Hilbert norm and inner product to define a notion of distance between the target posterior and the approximation. The primary focus will be on showing that the Frank-Wolfe Bayesian coreset algorithm \citep{Campbell2019} is equivalent to Frank-Wolfe kernel-herding \citep{Chen2010,Bach2012} when using the kernel \eqref{eq:BHS_kernel}. This novel relationship will be shown to be a consequence of Theorem \ref{thm:MMD_BHS}. A secondary focus will be on the iterative hard thresholding coreset algorithm \citep{Zhang2021} which will be related to MMD descent methods.

\subsection{Frank-Wolfe Bayesian coresets}
The goal of both the Frank-Wolfe Bayesian coreset algorithm and Frank-Wolfe kernel-herding is essentially is to approximate an element of a Hilbert space using a candidate set. In both cases, as the name suggests, the Frank-Wolfe optimisation method is used to find the best approximation over the set. To begin with, the approximation problem of concern and the Frank-Wolfe algorithm solution is outlined only in the generality needed to make the desired connections. For further discussion of the method consult \citet{Clarkson2010}.

For some Hilbert space $H$, fix some approximation target $f\in H$ and let $\{g_{n}\}_{n=1}^{N}\subset H$ be a set of elements which will be used to approximate the target. Let $\{\sigma_{n}\}_{n=1}^{N}\subset\bbR_{>0}$ be a set of positive real numbers, set $\sigma = \sum_{n=1}^{N}\sigma_{n}$ and define the polytope $\mathcal{W} = \{w\in\bbR^{N}\colon w\geq 0, \sum_{n=1}^{N}w_{n}\sigma_{n}\leq \sigma\}$, noting it has vertices $v_{n} = (\sigma/\sigma_{n})\mathbf{1}_{n}$ where $\mathbf{1}_{n}$ is the one-hot vector of all zeros except a $1$ in the $n$-th entry. The goal is to find
\begin{align}
	\argmin_{w\in\mathcal{W}}\frac{1}{2}\norm{f-g_{w}}_{H}^{2},\label{eq:FW_opt}
\end{align} 
where $g_{w} = \sum_{n=1}^{N}w_{n}g_{n}$. The Frank-Wolfe method to solve this problem involves iterative, conditional gradient updates. It looks at the residual of the current approximation, finds the direction most aligned with it, then performs line search in that direction. Specifically, at the $t$-th iteration one performs the update of the current guess $u_{t}\in\mathcal{W}$ as follows
\begin{align}
	\overline{u}_{t} & = \argmax_{v_{n}}\langle f-g_{u_{t}},g_{v_{n}}\rangle_{H}\label{eq:inner_opt}\\
	u_{t+1} & = (1-\rho_{t})u_{t} + \rho_{t}\overline{u}_{t},\nonumber
\end{align}
where $\rho_{t}$ is a scalar value calculated in closed form using $\overline{u}_{t},u_{t},f$ \citep{Bach2012} and $v_{n} = (\sigma/\sigma_{n})\mathbf{1}_{n}$ are the vertices of $\mathcal{W}$, meaning $g_{v_{n}} = (\sigma/\sigma_{n})g_{n}$. Only the vertices are optimised over in \eqref{eq:inner_opt} as this is equivalent to optimising over the whole polytope $\mathcal{W}$ since the function to optimise in \eqref{eq:inner_opt} is linear. This means that only one member of the approximating dictionary is added at each iteration, meaning after $T$ iterations the current solution is $T$-sparse. 

The Frank-Wolfe Hilbert coreset algorithm derived in \citet{Campbell2019} fits this template. A version which includes log-likelihood centering is now outlined, which was not used in \citet{Campbell2018,Campbell2019} though was later advocated by \citet{Campbell2019VI,Zhang2021} since it enforces the desired scale invariance property for the likelihood approximation. Set $H=L^{2}(\mu)$,  $f = \Psi_{\mu}(\pi), g_{n} = \phi_{\mu}(x_{n})$ and $\sigma_{n} = \norm{\phi_{\mu}(x_{n})}_{L^{2}(\mu)}$. With these substitutions the goal becomes finding 
\begin{align}
	\argmin_{w\in\mathcal{W}}\frac{1}{2}\left\lVert\Psi_{\mu}(\pi)-\sum_{n=1}^{N}w_{n}\phi_{\mu}(x_{n})\right\rVert_{L^{2}(\mu)}^{2} = \argmin_{w\in\mathcal{W}}\frac{1}{2}\left\lVert\Psi_{\mu}(\pi)-\Psi_{\mu}(\pi_{w})\right\rVert_{L^{2}(\mu)}^{2},\label{eq:FW_min}
\end{align}
where $\pi_{w}$ denotes the Bayesian coreset approximation \eqref{eq:coreset_aprx} with weights $w$ and $z_{n} = x_{n}\:\forall\:n$. The Frank-Wolfe iterations are
\begin{align}
	\overline{u}_{t}  = \frac{\sigma}{\sigma_{x_{n}}}\mathbf{1}_{n} \text{ where } n & = \argmin_{n\in[N]}\left\langle \Psi_{\mu}(\pi)-g_{u_{t}},\frac{\sigma}{\sigma_{x_{n}}}\phi_{\mu}(x_{n})\right\rangle_{L^{2}(\mu)}\label{eq:FW_update}\\
	u_{t+1} & = (1-\rho_{t})u_{t} + \rho_{t}\overline{u}_{t},\nonumber
\end{align}
where the full expression for the update of $\overline{u}_{t}$, see \eqref{eq:inner_opt}, has been written for clarity and $[N] = \{1,\ldots,N\}$. Note \citet[Equation 4.2]{Campbell2019} matches \eqref{eq:FW_opt} and \citet[Equation 4.4]{Campbell2019} matches \eqref{eq:inner_opt} once the centred log-likelihoods are substituted for the log-likelihoods. As \eqref{eq:FW_min} and \eqref{eq:FW_update} are written in terms of CLR transforms, the $L^{2}(\mu)$ inner products can be written in terms of the $B^{2}(\mu)$ inner product, see \eqref{eq:inner_prod}. This shows that the Frank-Wolfe Hilbert coreset algorithm can be written in the language of Bayes Hilbert spaces.

The Frank-Wolfe coreset algorithm will now be related to Frank-Wolfe kernel-herding. Kernel-herding is where the Hilbert space $H$ is set to be an RKHS and the algorithm is used in the context of solving the kernel pre-image problem, see Problem \ref{prob:MMD}. Let $k$ be the kernel \eqref{eq:BHS_kernel}, $H=\calH_{k}$ with approximation target $f=\sum_{n=1}^{N}k(x_{n},\cdot)$ and retain the same choice of $g_{n},\sigma_{n}$ as above. Kernel-herding aims to find
\begin{align}
	\argmin_{w\in\mathcal{W}}\frac{1}{2}\left\lVert\sum_{n=1}^{N}k(x_{n},\cdot) - \sum_{n=1}^{N}w_{n}k(x_{n},\cdot)\right\rVert_{k}^{2},\label{eq:FW_kernel_goal}
\end{align}
which, using Theorem \ref{thm:MMD_BHS}, is equal to \eqref{eq:FW_min}. The kernel-herding Frank-Wolfe updates are
\begin{align}
	\overline{u}_{t}  = \frac{\sigma}{\sigma_{x_{n}}}\mathbf{1}_{n} \text{ where } n & = \argmin_{n\in[N]}\left\langle \sum_{n=1}^{N}k(x_{n},\cdot)-g_{u_{t}},\frac{\sigma}{\sigma_{x_{n}}}k(x_{n},\cdot)\right\rangle_{k},\label{eq:FW_kernel_update}\\
	& = \argmin_{n\in[N]}\left\langle \Psi_{\mu}(\pi)-g_{u_{t}},\frac{\sigma}{\sigma_{x_{n}}}\phi_{\mu}(x_{n})\right\rangle_{L^{2}(\mu)},\label{eq:FW_kernel_con}\\
	u_{t+1} & = (1-\rho_{t})u_{t} + \rho_{t}\overline{u}_{t},\nonumber
\end{align}
where \eqref{eq:FW_kernel_con} is again by Theorem \ref{thm:MMD_BHS}. 

Due to the equivalences between the optimisation goals \eqref{eq:FW_kernel_goal} and \eqref{eq:FW_min} and updates \eqref{eq:FW_kernel_update} and \eqref{eq:FW_update} it can be concluded that the Frank-Wolfe Hilbert coreset algorithm is equivalent to Frank-Wolfe kernel-herding when one uses the kernel \eqref{eq:BHS_kernel} whose features map uses the CLR to map data to likelihoods in a Bayes Hilbert space. This identification immediately proposes some questions. For example, when kernel-herding is typically employed it does not used the subset of data requirement enforced in the Frank-Wolfe Bayesian coreset method, instead it chooses points that are anywhere on the domain which are called ``super-samples'' in \citet{Chen2010}. This leads to the question of how kernel-herding performs when used to construct Bayesian coresets without the limitation of the coreset points being a subset of the observed data points. Such a methodology was called \textit{pseudo-coresets} in \citet{Manousakas2020} when investigating coresets constructed using KL divergence rather than a Hilbert norm.
 
\subsection{Iterative hard thresholding Bayesian coresets}
The iterative hard thresholding (IHT) coreset algorithm \citep{Zhang2021} proposes to find the weights of a Bayesian coreset using gradient descent and then thresholds the weights to ensure they non-negative and sparse. More specifically, the problem under consideration is
\begin{align}
\argmin_{w\in \mathcal{W}_{M}}\frac{1}{2}\left\lVert \Psi_{\mu}(\pi) - \sum_{n=1}^{N}w_{n}\phi_{\mu}(x_{n})\right\rVert_{L^{2}(\mu)}^{2},\label{eq:IHT_prob}
\end{align}
where $\mathcal{W}_{M} = \{w\in\bbR^{N}\colon w\geq 0, \norm{w}_{0}\leq M\}$ is the set of vectors in $\bbR^{N}$ with non-negative weights and at most $M$ non-zero values. The parameter $M$ dictates the level of sparseness of the approximation. The IHT technique \citep{Zhang2021} solves this problem by performing gradient descent with respect to $w$ at each iteration and then projecting the weights to $\mathcal{W}_{M}$, to maintain a valid, sparse choice of weights. 

Using Theorem \ref{thm:MMD_BHS} is it made apparent that \eqref{eq:IHT_prob} is equivalent to finding a sparse, non-negative weights solution to the kernel pre-image problem, Problem \ref{prob:MMD}. Gradient descent methods were some of the first methods considered for the kernel pre-image problem \citep{Burges1996} although there were not studied extensively due to the computational considerations at the time. The IHT method utilises advances in accelerated optimisation methods to efficiently find choices of weights that result in a good approximation to the posterior. 

The link to the pre-image problem and the IHT method can be taken further by realising that the IHT method is optimising an MMD criterion. Therefore, IHT can be seen as performing MMD descent on the empirical measure that corresponds to the choice of points and weights. MMD descent is a method of empirical measure approximation that uses gradients to minimise an MMD objective \citep{Arbel2019}. There are related methods which use other kernel-based discrepancies \citep{Korba2021KSD,Xu2022}. This begs the question of the performance of these other kernel-based descent methods on the task of constructing Bayesian coresets when using the kernel \eqref{eq:BHS_kernel}.

\section{Kullback-Leibler Bayesian Coresets}\label{sec:KL_coresets}
The aim of this section is to provide a novel connection between Kullback-Leibler (KL) Bayesian coreset algorithms and Bayes Hilbert spaces. The former are Bayesian coreset algorithms which use the Kullback-Leibler divergence to measure the discrepancy between the target posterior and the Bayesian coreset posterior. The primary focus will be on showing that the quasi-newton KL Bayesian coreset algorithm \citep{Naik2022} is equivalent, up to hyper-parameter choices, to variational inference in a Bayes Hilbert space. This novel relationship will be shown to be a consequence of the description of variational inference techniques in terms of Bayes Hilbert spaces provided by \citet{Barfoot2023}. 

The KL divergence, see Section \ref{sec:BHS} is a divergence that is used often in variational inference \citep{Blei2017}. This is a measure approximation method where the KL divergence is minimised over a user chosen set of measures, called the variational family. The variational family is often parameterised using the parameters of a given distribution. For example, the variational family could be a mixture of Gaussian and the variational parameters would be the mixture weights and the means and covariances of each Gaussian. The optimisation to find the best approximation within the variational family typically involves gradient descent updates. This does not require samples from the target measure, an essential property in the case where the target measure is expensive to sample from, for example when it is the posterior after observing a large number of data points.  

Variational inference was framed in terms of Bayes Hilbert spaces by \citet{Barfoot2023} and this perspective is now outlined. Let $B^{2}(\mu)$ be a Bayes Hilbert space, $\nu\in B^{2}(\mu)$ be the target measure one wishes to approximate and $\{b_{n}\}_{n=1}^{N}\subset B^{2}(\mu)$ some set of $N$ elements which form the dictionary of the approximation. The variational family is $\{\oplus_{n=1}^{N}w_{n}\odot b_{n}\colon w\in\bbR^{N}\}$ with corresponding optimisation problem
\begin{align}
	\argmin_{w\in\bbR^{N}}\text{KL}\left(\bigoplus_{n=1}^{N}w_{n}\odot b_{n}\;\|\; \nu\right),\label{eq:Barfoot_KL}
\end{align}
where the notation $\oplus,\odot$ is defined in Section \ref{sec:BHS}. 

The key insight of \citet{Barfoot2023} is that solving this by a quasi-Newton gradient descent method with respect to $w$ is equivalent to doing iterative projections in a Bayes Hilbert space. Specifically, \citet[Equation 41]{Barfoot2023} provide the quasi-Newton update
\begin{align}
	w_{t+1} = w_{t} + G_{\nu_{t}}(\textbf{b})^{-1}\langle \textbf{b},\nu\ominus\nu_{t}\rangle_{B^{2}(\nu_{t})},\label{eq:Barfoot_update}
\end{align}
where $\nu_{t}$ is the approximation at iteration $t$ using weights $w_{t}$, $G_{\nu_{t}}(\textbf{b})$ is the $N\times N$ Gram matrix of the approximating dictionary $\{b_{n}\}_{n=1}^{N}$ with $(n,m)$-th entry $\langle b_{n},b_{m}\rangle_{B^{2}(\nu_{t})}$ and $\langle \textbf{b},\nu\ominus\nu_{t}\rangle_{B^{2}(\nu_{t})}$ is the $\bbR^{N}$ vector with $n$-th entry $\langle b_{n},\nu \ominus \nu_{t}\rangle_{B^{2}(\nu_{t})}$. See Section \ref{sec:BHS} for the definition of $\ominus$ and the Bayes Hilbert space inner product. It was noted that this quasi-Newton gradient update can be written as 
\begin{align*}
	w_{t+1} = G_{\nu_{t}}(\textbf{b})^{-1}\langle \textbf{b},\nu\rangle_{B^{2}(\nu_{t})},
\end{align*}
see \citet[Equation 43]{Barfoot2023}, meaning that at each iteration the target $\nu$ is being projected to the span of $\{b_{n}\}_{n=1}^{N}$ with respect to the $B^{2}(\nu_{t})$ inner product. This shows how the geometry of Bayes Hilbert spaces links with variational inference methods.

The connection between the quasi-Newton Bayesian coreset algorithm derived by \citet{Naik2022} and this Bayes Hilbert space perspective on variational inference is now made. The method of \citet{Naik2022} starts by setting the target measure to be $\pi$, the posterior, and choosing some $M\in\bbN$. Then a size $M$ subset from $\{x_{n}\}_{n=1}^{N}$ is uniformly sampled, which has its indices re-ordered and is denoted $\{x_{m}\}_{m=1}^{M}$. The weights for the log-likelihood terms $\{l_{x_{m}}\}_{m=1}^{M}$ are then optimised using a quasi-Newton iterations and thresholding the weights at each iteration to be non-negative. 

The connection to \citet{Barfoot2023} is made by setting $\nu = \pi$, $N=M$ and $b_{m} = l_{x_{m}}$ in \eqref{eq:Barfoot_KL}. This means the optimisation is
\begin{align*}
	\argmin_{w\in\bbR^{M}_{\geq 0}}\text{KL}\left(\bigoplus_{m=1}^{M}w_{m}\odot b_{m}\;\|\;\pi\right) = \argmin_{w\in\bbR^{M}_{\geq 0}}\KL{\pi_{w,x}}{\pi},
\end{align*}
which is equal to \eqref{eq:Barfoot_KL} except with an additional non-negative weight constraint. 

Reading off the quasi-Newton iteration \citep[Equation 13]{Naik2022} one sees it is equal, up to hyper-parameter choices, to the update \eqref{eq:Barfoot_update} since the covariance terms in \citet[Equation 13]{Naik2022} coincide with the $B^{2}(\nu_{t})$ inner product. The method of \citet{Naik2022} then thresholds the weights to ensure non-negativity. 

Overall, this shows that the quasi-Newton KL coreset method can be viewed in terms of iterative projections in a Bayes Hilbert space. This cross pollination of perspectives is productive since it provides a Bayes Hilbert space geometric interpretation to the quasi-Newton Bayesian coreset algorithm. In turn, the numerical and theoretical investigations of \citet{Naik2022} bolster the geometric perspective. Given that Bayes Hilbert spaces are tangent spaces in the information geometry sense \citep{Fukumizu2009,Pistone2013} the quasi-Newton coreset method \citep{Naik2022} can be seen as iterative updates of projecting in a direction in a tangent space and then updating the base location of the tangent space. This is also related to natural gradient descent as noted by \citet{Barfoot2023}. This perspective was introduced for the greedy KL coreset method in \citet{Campbell2019VI} and the results of this section show it applies in a broader sense via the Bayes Hilbert space perspective.
%!TEX root = ./main.tex

\section{Conclusion}\label{sec:Conclusion}
This manuscript has described Bayes Hilbert spaces and shown how they are appropriate spaces to perform posterior measure approximation. Along the way a novel bound relating the Bayes Hilbert space norm to common discrepancies was provided as well as a novel connection between kernel-based discrepancies and Bayes Hilbert spaces in the context of posterior approximation. Multiple Bayesian coreset methods were expressed in terms of Bayes Hilbert spaces which provided insight and new interpretations of how they work. 

Future avenues of research relating to Bayes Hilbert spaces, Bayesian coresets and posterior approximation are legion. Three main avenues are now outlined. First, an issue noted when defining Bayes Hilbert spaces was that they are in a sense too big, because they can contain infinite measures. The related work in information geometry \citep{Fukumizu2009} defines a related set of measures that is homeomorphic to a subset of a Hilbert space that has stronger integrability conditions, but the set of measures is not a vector space whereas the Bayes Hilbert space is. Therefore, a question is what is an appropriate space with stricter integrability conditions to maintain finite measures while also having a vector space structure. Second, as noted in \citet{Boogaart2014} it is straight forward to construct a basis for the Bayes Hilbert space by using orthogonal polynomials. The specific case of a Gaussian base measure and Hermite polynomials was studied in \citet{Barfoot2023}. There have been many recent advances in the study of sparse, high-dimensional approximation using sparse polynomials \citep{Adcock2022}. This begs the question of how these methods can be applied to the Bayes Hilbert space. Using polynomials for likelihood approximation in the context of Bayesian computation was studied in \citet{Huggins2017} and combining this approach with the aforementioned advances could provide further advantages. Finally, there have been many recent innovations in distribution compression and quantisation of empirical measures \citep{Riabiz2022,Dwivedi2021}. Given that Theorem \ref{thm:MMD_BHS} links this problem to the Bayesian coreset problem the question immediately arises as to how to leverage thee innovations in the context of posterior approximation, in particular in the case where the points to form the likelihood $\{z_{m}\}_{m=1}^{M}$ are not a subset of the observed data. This pseudo-coreset method has only been explored for Bayesian coresets in the KL coreset setting \citep{Manousakas2020} but it is a natural thing to do in distribution compression setting \citep{Chen2010}.

\section{Appendix}

\subsection{Proof of Theorem \ref{thm:bounds}}
The proof of this result is simply an adaptation of Theorem 5, Theorem 11 and Theorem 14 from \citet{Sprungk2020}. All the proofs leverage a local Lipschitz continuity which follows from the assumption involving the bound $B$ on the CLR transforms. 

Recall that the measures in question $\eta,\nu$ are written as 
	\begin{align}
	\frac{\diff\eta}{\diff\mu} & = Z_{\eta}^{-1}\exp(\Psi_{\mu}(\eta))\qquad Z_{\eta} = \bbE_{\mu}[\exp(\Psi_{\mu}(\eta))]\\
	\frac{\diff\nu}{\diff\mu} & = Z_{\nu}^{-1}\exp(\Psi_{\mu}(\nu))\qquad Z_{\nu} = \bbE_{\mu}[\exp(\Psi_{\mu}(\nu))].
\end{align}
By Jensen's inequality, this means that $Z_{\eta}\geq \exp(\bbE_{\mu}[\Psi_{\mu}(\eta)]) = 1$ since $\Psi_{\mu}(\eta)\in L^{2}_{0}(\mu)$, the same lower bound applies to $Z_{\nu}$ too. Therefore, $Z_{\eta}^{-1},Z_{\nu}^{-1}\leq 1$.

\subsubsection*{Hellinger Distance}

First the Hellinger distance bound is derived. Starting from the definition \eqref{eq:Hellinger} and using $\mu$ as the base measure
\begin{align*}
	2\text{H}(\eta,\nu)^{2} & = \int_{\Theta}\left(e^{\frac{1}{2}\Psi_{\mu}(\eta)}Z_{\eta}^{-\frac{1}{2}} - e^{\frac{1}{2}\Psi_{\mu}(\nu)}Z_{\nu}^{-\frac{1}{2}}\right)^{2}\diff\mu\\
	& \leq 2\int_{\Theta}\left(e^{\frac{1}{2}\Psi_{\mu}(\eta)}Z_{\eta}^{-\frac{1}{2}} -e^{\frac{1}{2}\Psi_{\mu}(\nu)}Z_{\eta}^{-\frac{1}{2}}  \right)^{2}\diff\mu\\
	& + 2\int_{\Theta}\left(e^{\frac{1}{2}\Psi_{\mu}(\nu)}Z_{\eta}^{-\frac{1}{2}} -e^{\frac{1}{2}\Psi_{\mu}(\nu)}Z_{\nu}^{-\frac{1}{2}}  \right)^{2}\diff\mu\\
	& \eqqcolon I_{1} + I_{2},
\end{align*}
where the inequality $(a-b)^{2}\leq 2(a-c)^{2} + 2(c-b)^{2}$ has been used which is a result of the triangle inequality. The terms $I_{1},I_{2}$ will now be upper bounded.
\begin{align}
	I_{1} & = 2Z_{\eta}^{-1}\int_{\Theta}\left(e^{\frac{1}{2}\Psi_{\mu}(\eta)}-e^{\frac{1}{2}\Psi_{\mu}(\nu)}\right)^{2}\diff\mu \nonumber\\
	& \leq 2e^{B}\int_{\Theta}\left(\frac{\Psi_{\mu}(\eta)}{2} - \frac{\Psi_{\mu}(\nu)}{2}\right)^{2}\diff\mu \label{eq:loc_lip_H}\\
	& = \frac{1}{2}e^{B}\norm{\Psi_{\mu}(\eta)-\Psi_{\mu}(\nu)}_{L^{2}(\mu)}^{2} = \frac{1}{2}e^{B}\norm{\eta-\nu}_{B^{2}(\mu)}^{2},\nonumber
\end{align}
where \eqref{eq:loc_lip_H} uses $\lvert e^{x}-e^{y}\rvert\leq e^{\max(x,y)}\lvert x-y\rvert\:\forall\:x,y,\in\bbR$ and the assumed upper bound $B$ on $\Psi_{\mu}(\eta),\Psi_{\mu}(\nu)$. To bound $I_{2}$ first note that by the definition of the normalisation constants, 
\begin{align*}
	I_{2} & = 2\int_{\Theta}\left(e^{\frac{1}{2}\Psi_{\mu}(\nu)}Z_{\eta}^{-\frac{1}{2}} -e^{\frac{1}{2}\Psi_{\mu}(\nu)}Z_{\nu}^{-\frac{1}{2}}  \right)^{2}\diff\mu\\
	& = 2Z_{\eta}^{-1}\left(Z_{\eta}^{\frac{1}{2}} - Z_{\nu}^{\frac{1}{2}}\right)^{2} \leq 2\left(Z_{\eta}^{\frac{1}{2}} - Z_{\nu}^{\frac{1}{2}}\right)^{2}.
\end{align*}
Now, $\lvert x^{1/2}-y^{1/2}\rvert\leq \frac{1}{2}\min(x,y)^{-1/2}\lvert x-y\rvert\:\forall\:x,y > 0$ gives
\begin{align*}
	I_{2}\leq \frac{1}{\min(Z_{\eta},Z_{\nu})}\lvert Z_{\eta}-Z_{\nu}\rvert^{2}\leq\frac{1}{2}\lvert Z_{\eta}-Z_{\nu}\rvert^{2}.
\end{align*}
Finally,
\begin{align*}
	\lvert Z_{\eta}-Z_{\nu}\rvert^{2}\leq \int_{\Theta}\left\lvert e^{\Psi_{\mu}(\eta)}-e^{\Psi_{\mu}(\nu)}\right\rvert^{2}\diff\mu\leq e^{2B}\norm{\Psi_{\mu}(\eta)-\Psi_{\mu}(\nu)}_{L^{2}(\mu)}^{2} = e^{2B}\norm{\eta-\nu}_{B^{2}(\mu)}^{2},
\end{align*}
where the same bounding technique as \eqref{eq:loc_lip_H} was used. Putting this together gives
\begin{align*}
	\text{H}(\eta,\nu)^{2}\leq \frac{1}{4}\left(e^{B}+e^{2B}\right)\norm{\eta-\nu}_{B^{2}(\mu)}^{2},
\end{align*}
as desired. 

\subsubsection*{KL Divergence}
The case for KL divergence is now covered. First note 
\begin{align*}
	\frac{\diff \eta}{\diff \nu} = \frac{\diff \eta}{\diff\mu}\frac{\diff\mu}{\diff\nu} = Z_{\eta}^{-1}Z_{\nu}e^{\Psi_{\mu}(\eta)-\Psi_{\mu}(\nu)},
\end{align*}
therefore
\begin{align*}
	\KL{\eta}{\nu}=\int_{\Theta}\log\frac{\diff\eta}{\diff\nu}\diff\eta & \leq \left\lvert\log Z_{\eta} - \log Z_{\nu}\right\rvert + \int_{\Theta}\left\lvert\Psi_{\mu}(\eta)-\Psi_{\mu}(\nu)\right\rvert Z_{\eta}^{-1}e^{\Psi_{\mu}(\eta)}\diff\mu\\
	& \eqqcolon I_{1} + I_{2}.
\end{align*}
Using $\lvert \log x-\log y\rvert\leq \min(x,y)^{-1}\lvert x-y\rvert\:\forall x,y > 0$ gives
\begin{align*}
	I_{1}& \leq \min(Z_{\eta},Z_{\nu})^{-1}\lvert Z_{\eta}-Z_{\nu}\rvert\\
	& \leq e^{B}\norm{\Psi_{\mu}(\eta)-\Psi_{\mu}(\nu)}_{L^{1}(\mu)}\\
	& \leq e^{B}\norm{\Psi_{\mu}(\eta)-\Psi_{\mu}(\nu)}_{L^{2}(\mu)} = e^{B}\norm{\eta-\nu}_{B^{2}(\mu)},
\end{align*}
where the local Lipschitz bound $\lvert e^{x}-e^{y}\rvert\leq e^{\max(x,y)}\lvert x-y\rvert\:\forall\:x,y,\in\bbR$ was used to get the $L^{1}(\mu)$ norm upper bound, similar to its use in \eqref{eq:loc_lip_H}. The $L^{2}(\mu)$ norm upper bounds the $L^{1}(\mu)$ norm since $\mu$ is a probability measure. If $\mu$ was assumed only to be a finite measure then there would be an extra factor of $\mu(\Theta)$ in the bounds. Bounding $I_{2}$ is done in a straightforward way
\begin{align*}
	I_{2} = \int_{\Theta}\left\lvert\Psi_{\mu}(\eta)-\Psi_{\mu}(\nu)\right\rvert Z_{\eta}^{-1}e^{\Psi_{\mu}(\eta)}\diff\mu & \leq e^{B}\norm{\Psi_{\mu}(\eta)-\Psi_{\mu}(\nu)}_{L^{1}(\mu)}\\& \leq e^{B}\norm{\Psi_{\mu}(\eta)-\Psi_{\mu}(\nu)}_{L^{2}(\mu)} = e^{B}\norm{\eta-\nu}_{B^{2}(\mu)}.
\end{align*}
Putting $I_{1},I_{2}$ together gives the desired bound for the KL divergence. 

\subsubsection*{Wasserstein-$1$ Distance}
Finally, the Wasserstein-$1$ case is dealt with. Take any $\theta_{0}\in\Theta$ then by simply shifting the values it suffices to consider the supremum over functions in $\text{Lip}(1)$ such that $f(\theta_{0}) = 0$. For such functions $\lvert f(\theta)\rvert\leq d_{\Theta}(\theta,\theta_{0})$. Following \citet[Theorem 14]{Sprungk2020},
\begin{align*}
	\left\lvert\int_{\Theta}f\diff\eta-\int_{\Theta}f\diff\nu\right\rvert & = \left\lvert\int_{\Theta}f\cdot\left(Z_{\eta}^{-1}e^{\Psi_{\mu}(\eta)}-Z_{\nu}^{-1}e^{\Psi_{\mu}(\nu)}\right)\diff\mu\right\rvert\\
	& \leq \left\lvert Z_{\eta}^{-1}-Z_{\nu}^{-1}\right\rvert\left\lvert\int_{\Theta}fe^{\Psi_{\mu}(\eta)}\diff\mu\right\rvert + \left\vert Z_{\nu}^{-1}\int_{\Theta}f\cdot\left(e^{\Psi_{\mu}(\eta)}-e^{\Psi_{\mu}(\nu)}\right)\diff\mu\right\rvert\\
	& \eqqcolon I_{1} + I_{2},
\end{align*}
where the inequality is by the triangle inequality. To bound $I_{1}$ note
\begin{align*}
	\left\lvert Z_{\eta}^{-1}-Z_{\nu}^{-1}\right\rvert = \frac{\left\lvert Z_{\eta}-Z_{\nu}\right\rvert}{Z_{\eta}Z_{\nu}}\leq \left\lvert Z_{\eta}-Z_{\nu}\right\rvert\leq e^{B}\norm{\Psi_{\mu}(\eta)-\Psi_{\mu}(\nu)}_{L^{1}(\mu)},
\end{align*}
where the first inequality is by the lower bound of $1$ on the normalising constants and the second by using again the local Lipschitz result, as was done in \eqref{eq:loc_lip_H}. Using this gives
\begin{align*}
	I_{1}\leq \int_{\Theta}\left\lvert f e^{\Psi_{\mu}(\eta)}\right\rvert\diff\mu\cdot e^{B}\norm{\Psi_{\mu}(\eta)-\Psi_{\mu}(\nu)}_{L^{1}(\mu)} \leq e^{2B}\norm{\mu}_{\calP^{1}}\norm{\Psi_{\mu}(\eta)-\Psi_{\mu}(\nu)}_{L^{1}(\mu)},
\end{align*}
where the second inequality is by the bounds on $\lvert f(\theta)\rvert$, $\Psi_{\mu}(\eta)$ and the way the choice of $\theta_{0}$ was arbitrary, which allows the infimum bound $norm{\mu}_{\calP^{1}}$.

Bounding $I_{2}$ is largely similar,
\begin{align*}
	I_{2}\leq \int_{\Theta}\lvert d_{\Theta}(\theta,\theta_{0})\rvert \left\lvert e^{\Psi_{\mu}(\eta)}-e^{\Psi_{\mu}(\nu)}\right\rvert\diff\mu\leq  e^{B}\norm{\mu}_{\calP^{2}}\int_{\Theta} \norm{\Psi_{\mu}(\eta)-\Psi_{\mu}(\nu)}_{L^{2}(\mu)},
\end{align*}
where the second inequality is by the Cauchy-Schwarz inequality and the local Lipschitz bound again. Overall, using the Cauchy-Schwartz inequality along with how $\mu$ is a probability measure to bound the $L^{1}$ norm with the $L^{2}$ norm gives
\begin{align*}
	\text{W}_{1}(\eta,\nu)\leq \left(e^{B}+e^{2B}\right)\norm{\mu}_{\calP^{2}}\norm{\Psi_{\mu}(\eta)-\Psi_{\mu}(\nu)}_{L^{2}(\mu)} = \left(e^{B}+e^{2B}\right)\norm{\mu}_{\calP^{2}}\norm{\eta-\nu}_{B^{2}(\mu)}.
\end{align*}

\subsection{Proof of Corollary \ref{cor:concentration}}
The proof is a simple application of \citet[Theorem 2]{Schneider2016}. First, note that by Theorem \ref{thm:MMD_BHS} 
\begin{align*}
	\left\lVert\pi_{M}-\pi\right\rVert_{B^{2}(\mu)} = \text{MMD}_{k}(P_{M},P_{N}) = \left\lVert \Phi_{k}P_{M} - \Phi_{k}P_{N}\right\rVert_{k} ,
\end{align*}
where $k$ is the kernel \eqref{eq:BHS_kernel}, $P_{M} = \frac{N}{M}\sum_{m=1}^{M}z_{m}$ and $P_{N} = \sum_{n=1}^{N}\delta_{x_{n}}$. Therefore, the quantity to be bounded is the same as in \citet[Theorem 2]{Schneider2016} up to an additional scaling of $N$ since the referenced result involves the normalised empirical measures $\widetilde{P}_{M} = \frac{1}{N}P_{M}, \widetilde{P}_{N} = \frac{1}{N}P_{N}$. This means the bound in \citet[Theorem 2]{Schneider2016} implies for any $\varepsilon_{0}$
\begin{align*}
	\bbP\left(\left\lVert \Phi_{k}\widetilde{P}_{M} - \Phi_{k}\widetilde{P}_{N}\right\rVert_{k}\geq \varepsilon_{0}\right) = \bbP(\left\lVert \Phi_{k}P_{M} - \Phi_{k}P_{N}\right\rVert_{k}\geq N\varepsilon_{0})\leq 2\exp\left(-\frac{M\varepsilon_{0}^{2}}{8\gamma^{2}(1-(M-1)/N)}\right),
\end{align*}
and then the final step is simply substituting $\varepsilon = N\varepsilon_{0}$ and  rearranging to obtain the bound purely in terms of $\delta$. 

{
\bibliographystyle{abbrvnat}
\bibliography{BHS_refs.bib}

\begin{thebibliography}{76}
\providecommand{\natexlab}[1]{#1}
\providecommand{\url}[1]{\texttt{#1}}
\expandafter\ifx\csname urlstyle\endcsname\relax
  \providecommand{\doi}[1]{doi: #1}\else
  \providecommand{\doi}{doi: \begingroup \urlstyle{rm}\Url}\fi

\bibitem[Adcock et~al.(2022)Adcock, Brugiapaglia, and Webster]{Adcock2022}
B.~Adcock, S.~Brugiapaglia, and C.~G. Webster.
\newblock \emph{Sparse Polynomial Approximation of High-Dimensional Functions}.
\newblock Society for Industrial and Applied Mathematics, 2022.

\bibitem[Aitchison(1982)]{Aitchson1982}
J.~Aitchison.
\newblock The statistical analysis of compositional data.
\newblock \emph{Journal of the Royal Statistical Society. Series B.
  Methodological}, 44\penalty0 (2):\penalty0 139--177, 1982.
\newblock With discussion.

\bibitem[Aitchison(1986)]{Aitchison1986}
J.~Aitchison.
\newblock \emph{The Statistical Analysis of Compositional Data}.
\newblock Chapman and Hall, 1986.

\bibitem[Anastasiou et~al.(2023)Anastasiou, Barp, Briol, Ebner, Gaunt,
  Ghaderinezhad, Gorham, Gretton, Ley, Liu, Mackey, Oates, Reinert, and
  Swan]{Anastasiou2023}
A.~Anastasiou, A.~Barp, F.-X. Briol, B.~Ebner, R.~E. Gaunt, F.~Ghaderinezhad,
  J.~Gorham, A.~Gretton, C.~Ley, Q.~Liu, L.~Mackey, C.~J. Oates, G.~Reinert,
  and Y.~Swan.
\newblock Stein’s method meets computational statistics: A review of some
  recent developments.
\newblock \emph{Statistical Science}, 38\penalty0 (1):\penalty0 120 -- 139,
  2023.

\bibitem[Arbel et~al.(2019)Arbel, Korba, SALIM, and Gretton]{Arbel2019}
M.~Arbel, A.~Korba, A.~SALIM, and A.~Gretton.
\newblock Maximum mean discrepancy gradient flow.
\newblock In \emph{Advances in Neural Information Processing Systems},
  volume~32, 2019.

\bibitem[Bach et~al.(2012)Bach, Lacoste-Julien, and Obozinski]{Bach2012}
F.~Bach, S.~Lacoste-Julien, and G.~Obozinski.
\newblock On the equivalence between herding and conditional gradient
  algorithms.
\newblock In \emph{Proceedings of the 29th International Coference on
  International Conference on Machine Learning}, page 1355–1362, 2012.

\bibitem[Barfoot and D’Eleuterio(2023)]{Barfoot2023}
T.~D. Barfoot and G.~M.~T. D’Eleuterio.
\newblock Variational inference as iterative projection in a {B}ayesian
  {H}ilbert space with application to robotic state estimation.
\newblock \emph{Robotica}, 41\penalty0 (2):\penalty0 632–667, 2023.

\bibitem[Berlinet and Thomas-Agnan(2004)]{Berlinet2004}
A.~Berlinet and C.~Thomas-Agnan.
\newblock \emph{Reproducing Kernel Hilbert Spaces in Probability and
  Statistics}.
\newblock Springer {US}, 2004.

\bibitem[Bierkens et~al.(2019)Bierkens, Fearnhead, and Roberts]{Bierkins2019}
J.~Bierkens, P.~Fearnhead, and G.~Roberts.
\newblock {The Zig-Zag process and super-efficient sampling for Bayesian
  analysis of big data}.
\newblock \emph{The Annals of Statistics}, 47\penalty0 (3):\penalty0 1288 --
  1320, 2019.

\bibitem[Blei et~al.(2017)Blei, Kucukelbir, and McAuliffe]{Blei2017}
D.~M. Blei, A.~Kucukelbir, and J.~D. McAuliffe.
\newblock Variational inference: A review for statisticians.
\newblock \emph{Journal of the American Statistical Association}, 112\penalty0
  (518):\penalty0 859--877, 2017.

\bibitem[Bogachev(2007)]{Bogachev2007}
V.~I. Bogachev.
\newblock \emph{Measure Theory}.
\newblock Springer Berlin Heidelberg, 2007.

\bibitem[Burges(1996)]{Burges1996}
C.~J.~C. Burges.
\newblock Simplified support vector decision rules.
\newblock In \emph{International Conference on Machine Learning}, 1996.

\bibitem[Caflisch(1998)]{Caflisch1998}
R.~E. Caflisch.
\newblock {Monte Carlo} and {quasi-Monte Carlo} methods.
\newblock \emph{Acta Numerica}, 7:\penalty0 1–49, 1998.

\bibitem[Campbell and Beronov(2019)]{Campbell2019VI}
T.~Campbell and B.~Beronov.
\newblock Sparse variational inference: Bayesian coresets from scratch.
\newblock In \emph{Advances in Neural Information Processing Systems},
  volume~32. Curran Associates, Inc., 2019.

\bibitem[Campbell and Broderick(2018)]{Campbell2018}
T.~Campbell and T.~Broderick.
\newblock {B}ayesian coreset construction via greedy iterative geodesic ascent.
\newblock In \emph{Proceedings of the 35th International Conference on Machine
  Learning}, volume~80 of \emph{Proceedings of Machine Learning Research},
  pages 698--706, 2018.

\bibitem[Campbell and Broderick(2019)]{Campbell2019}
T.~Campbell and T.~Broderick.
\newblock Automated scalable bayesian inference via hilbert coresets.
\newblock \emph{Journal of Machine Learning Research}, 20\penalty0
  (15):\penalty0 1--38, 2019.

\bibitem[Cena and Pistone(2006)]{Cena2006}
A.~Cena and G.~Pistone.
\newblock Exponential statistical manifold.
\newblock \emph{Annals of the Institute of Statistical Mathematics},
  59\penalty0 (1):\penalty0 27--56, 2006.

\bibitem[Chen et~al.(2010)Chen, Welling, and Smola]{Chen2010}
Y.~Chen, M.~Welling, and A.~Smola.
\newblock Super-samples from kernel herding.
\newblock In \emph{Proceedings of the Twenty-Sixth Conference on Uncertainty in
  Artificial Intelligence}, page 109–116, 2010.

\bibitem[Cherief-Abdellatif and Alquier(2020)]{Cherief2020}
B.-E. Cherief-Abdellatif and P.~Alquier.
\newblock {MMD-Bayes}: Robust bayesian estimation via maximum mean discrepancy.
\newblock In \emph{Proceedings of The 2nd Symposium on Advances in Approximate
  Bayesian Inference}, volume 118, pages 1--21, 2020.

\bibitem[Christmann and Steinwart(2008)]{Steinwart2008}
A.~Christmann and I.~Steinwart.
\newblock \emph{Support Vector Machines}.
\newblock Springer New York, 2008.

\bibitem[Clarkson(2010)]{Clarkson2010}
K.~L. Clarkson.
\newblock Coresets, sparse greedy approximation, and the {Frank-Wolfe}
  algorithm.
\newblock \emph{ACM Transactions on Algorithms}, 6\penalty0 (4), 2010.

\bibitem[Dang et~al.(2019)Dang, Quiroz, Kohn, Tran, and Villani]{Dang2019}
K.-D. Dang, M.~Quiroz, R.~Kohn, M.-N. Tran, and M.~Villani.
\newblock {Hamiltonian Monte Carlo} with energy conserving subsampling.
\newblock \emph{Journal of Machine Learning Research}, 20\penalty0
  (100):\penalty0 1--31, 2019.

\bibitem[Dick et~al.(2013)Dick, Kuo, and Sloan]{Dick2013}
J.~Dick, F.~Y. Kuo, and I.~H. Sloan.
\newblock High-dimensional integration: The {quasi-Monte Carlo} way.
\newblock \emph{Acta Numerica}, 22:\penalty0 133–288, 2013.

\bibitem[Dwivedi and Mackey(2021)]{Dwivedi2021}
R.~Dwivedi and L.~Mackey.
\newblock Kernel thinning.
\newblock In \emph{Proceedings of Thirty Fourth Conference on Learning Theory},
  volume 134, pages 1753--1753, 2021.

\bibitem[Egozcue et~al.(2006)Egozcue, D\'{\i}az-Barrero, and
  Pawlowsky-Glahn]{Egozcue2006}
J.~J. Egozcue, J.~L. D\'{\i}az-Barrero, and V.~Pawlowsky-Glahn.
\newblock Hilbert space of probability density functions based on {A}itchison
  geometry.
\newblock \emph{Acta Mathematica Sinica (English Series)}, 22\penalty0
  (4):\penalty0 1175--1182, 2006.

\bibitem[Erb and Ay(2021)]{Erb2021}
I.~Erb and N.~Ay.
\newblock The information-geometric perspective of compositional data analysis.
\newblock In \emph{Advances in Compositional Data Analysis}, pages 21--43.
  Springer International Publishing, 2021.

\bibitem[Filzmoser et~al.(2018)Filzmoser, Hron, and Templ]{Filzmoser2018}
P.~Filzmoser, K.~Hron, and M.~Templ.
\newblock \emph{Applied Compositional Data Analysis}.
\newblock Springer International Publishing, 2018.

\bibitem[Fukumizu(2009)]{Fukumizu2009}
K.~Fukumizu.
\newblock \emph{Exponential manifold by reproducing kernel Hilbert spaces},
  page 291–306.
\newblock Cambridge University Press, 2009.

\bibitem[Fukumizu et~al.(2013)Fukumizu, Song, and Gretton]{Fukumizu2013}
K.~Fukumizu, L.~Song, and A.~Gretton.
\newblock Kernel {Bayes'} rule: Bayesian inference with positive definite
  kernels.
\newblock \emph{Journal of Machine Learning Research}, 14\penalty0
  (118):\penalty0 3753--3783, 2013.

\bibitem[Gelman et~al.(2013)Gelman, Carlin, Stern, Dunson, Vehtari, and
  Rubin]{Gelman2013}
A.~Gelman, J.~B. Carlin, H.~S. Stern, D.~B. Dunson, A.~Vehtari, and D.~B.
  Rubin.
\newblock \emph{Bayesian Data Analysis}.
\newblock Chapman and Hall/{CRC}, 2013.

\bibitem[Graf and Luschgy(2000)]{Graf2000}
S.~Graf and H.~Luschgy.
\newblock \emph{Foundations of Quantization for Probability Distributions}.
\newblock Springer Berlin Heidelberg, 2000.

\bibitem[Greenacre et~al.(2022)Greenacre, Grunsky, Bacon-Shone, Erb, and
  Quinn]{Greenacre2022}
M.~Greenacre, E.~Grunsky, J.~Bacon-Shone, I.~Erb, and T.~Quinn.
\newblock Aitchison's compositional data analysis 40 years on: A reappraisal.
\newblock \emph{arXiv:2201.05197}, 2022.

\bibitem[Gretton et~al.(2006)Gretton, Borgwardt, Rasch, Sch\"{o}lkopf, and
  Smola]{Gretton2006}
A.~Gretton, K.~Borgwardt, M.~Rasch, B.~Sch\"{o}lkopf, and A.~Smola.
\newblock A kernel method for the two-sample-problem.
\newblock In \emph{Advances in Neural Information Processing Systems},
  volume~19. MIT Press, 2006.

\bibitem[Gretton et~al.(2012)Gretton, Borgwardt, Rasch, Sch\"{o}lkopf, and
  Smola]{Gretton2012}
A.~Gretton, K.~M. Borgwardt, M.~J. Rasch, B.~Sch\"{o}lkopf, and A.~Smola.
\newblock A kernel two-sample test.
\newblock \emph{Journal of Machine Learning Research}, 13\penalty0
  (1):\penalty0 723--773, 2012.

\bibitem[Guilbart(1978)]{Guilbart1978}
C.~Guilbart.
\newblock \emph{Etude des Produits Scalaires sur l’Espace des Mesures:
  Estimation par Projections}.
\newblock PhD thesis, Université des Sciences et Techniques de Lille, 1978.

\bibitem[Hoffman et~al.(2013)Hoffman, Blei, Wang, and Paisley]{Hoffman2013}
M.~D. Hoffman, D.~M. Blei, C.~Wang, and J.~Paisley.
\newblock Stochastic variational inference.
\newblock \emph{Journal of Machine Learning Research}, 14\penalty0
  (40):\penalty0 1303--1347, 2013.

\bibitem[Hron et~al.(2022)Hron, Machalov{\'{a}}, and Menafoglio]{Hron2022}
K.~Hron, J.~Machalov{\'{a}}, and A.~Menafoglio.
\newblock Bivariate densities in {Bayes} spaces: Orthogonal decomposition and
  spline representation.
\newblock \emph{Statistical Papers}, 2022.

\bibitem[Huggins et~al.(2016)Huggins, Campbell, and Broderick]{Huggins2016}
J.~Huggins, T.~Campbell, and T.~Broderick.
\newblock Coresets for scalable {Bayesian} logistic regression.
\newblock In \emph{Advances in Neural Information Processing Systems},
  volume~29, 2016.

\bibitem[Huggins et~al.(2017)Huggins, Adams, and Broderick]{Huggins2017}
J.~Huggins, R.~P. Adams, and T.~Broderick.
\newblock {PASS-GLM}: Polynomial approximate sufficient statistics for scalable
  {Bayesian} {GLM} inference.
\newblock In \emph{Advances in Neural Information Processing Systems},
  volume~30, 2017.

\bibitem[Huggins et~al.(2018)Huggins, Campbell, Kasprzak, and
  Broderick]{Huggins2018}
J.~H. Huggins, T.~Campbell, M.~Kasprzak, and T.~Broderick.
\newblock Practical bounds on the error of {B}ayesian posterior approximations:
  A nonasymptotic approach.
\newblock \emph{arXiv:1809.09505}, 2018.

\bibitem[Korba et~al.(2021)Korba, Aubin-Frankowski, Majewski, and
  Ablin]{Korba2021KSD}
A.~Korba, P.-C. Aubin-Frankowski, S.~Majewski, and P.~Ablin.
\newblock Kernel stein discrepancy descent.
\newblock In \emph{Proceedings of the 38th International Conference on Machine
  Learning}, volume 139, pages 5719--5730, 2021.

\bibitem[Kroese et~al.(2011)Kroese, Taimre, and Botev]{Kroese2011}
D.~P. Kroese, T.~Taimre, and Z.~I. Botev.
\newblock \emph{Handbook of Monte Carlo Methods}.
\newblock Wiley, 2011.

\bibitem[Kwok and Tsang(2004)]{Kwok2004}
J.-Y. Kwok and I.-H. Tsang.
\newblock The pre-image problem in kernel methods.
\newblock \emph{IEEE Transactions on Neural Networks}, 15\penalty0
  (6):\penalty0 1517--1525, 2004.

\bibitem[Li et~al.(2017)Li, Chang, Cheng, Yang, and Poczos]{Li2017}
C.-L. Li, W.-C. Chang, Y.~Cheng, Y.~Yang, and B.~Poczos.
\newblock {MMD GAN}: Towards deeper understanding of moment matching network.
\newblock In \emph{Advances in Neural Information Processing Systems},
  volume~30, 2017.

\bibitem[Maier et~al.(2021)Maier, St\"{o}cker, Fitzenberger, and
  Greven]{Maier2021}
E.-M. Maier, A.~St\"{o}cker, B.~Fitzenberger, and S.~Greven.
\newblock Additive density-on-scalar regression in {Bayes Hilbert} spaces with
  an application to gender economics.
\newblock \emph{arXiv:2110.11771}, 2021.

\bibitem[Mak and Joseph(2018)]{Mak2018}
S.~Mak and V.~R. Joseph.
\newblock Support points.
\newblock \emph{The Annals of Statistics}, 46\penalty0 (6A):\penalty0
  2562--2592, 2018.

\bibitem[Manousakas et~al.(2020)Manousakas, Xu, Mascolo, and
  Campbell]{Manousakas2020}
D.~Manousakas, Z.~Xu, C.~Mascolo, and T.~Campbell.
\newblock Bayesian pseudocoresets.
\newblock In \emph{Advances in Neural Information Processing Systems},
  volume~33, pages 14950--14960, 2020.

\bibitem[Mateu and Giraldo(2021)]{Mateu2021}
J.~Mateu and R.~Giraldo, editors.
\newblock \emph{Geostatistical Functional Data Analysis}.
\newblock Wiley, 2021.

\bibitem[Mika et~al.(1998)Mika, Sch\"{o}lkopf, Smola, M\"{u}ller, Scholz, and
  R\"{a}tsch]{Mika1998}
S.~Mika, B.~Sch\"{o}lkopf, A.~Smola, K.-R. M\"{u}ller, M.~Scholz, and
  G.~R\"{a}tsch.
\newblock Kernel {PCA} and de-noising in feature spaces.
\newblock In \emph{Advances in Neural Information Processing Systems},
  volume~11, 1998.

\bibitem[Muandet et~al.(2017)Muandet, Fukumizu, Sriperumbudur, and
  Schölkopf]{Muandet2017}
K.~Muandet, K.~Fukumizu, B.~Sriperumbudur, and B.~Schölkopf.
\newblock Kernel mean embedding of distributions: A review and beyond.
\newblock \emph{Foundations and Trends® in Machine Learning}, 10\penalty0
  (1-2):\penalty0 1--141, 2017.

\bibitem[Naik et~al.(2022)Naik, Rousseau, and Campbell]{Naik2022}
C.~Naik, J.~Rousseau, and T.~Campbell.
\newblock Fast {Bayesian} coresets via subsampling and quasi-newton refinement.
\newblock In \emph{Advances in Neural Information Processing Systems}, 2022.

\bibitem[Newton(2012)]{Newton2012}
N.~J. Newton.
\newblock An infinite-dimensional statistical manifold modelled on {Hilbert}
  space.
\newblock \emph{Journal of Functional Analysis}, 263\penalty0 (6):\penalty0
  1661--1681, 2012.

\bibitem[Paulsen and Raghupathi(2016)]{Paulsen2016}
V.~I. Paulsen and M.~Raghupathi.
\newblock \emph{An Introduction to the Theory of Reproducing Kernel {H}ilbert
  Spaces}, volume 152 of \emph{Cambridge Studies in Advanced Mathematics}.
\newblock Cambridge University Press, Cambridge, 2016.

\bibitem[Pawlowsky-Glahn and Buccianti(2011)]{Glahn2011}
V.~Pawlowsky-Glahn and A.~Buccianti, editors.
\newblock \emph{Compositional Data Analysis}.
\newblock Wiley, 2011.

\bibitem[Petersen et~al.(2022)Petersen, Zhang, and Kokoszka]{Petersen2022}
A.~Petersen, C.~Zhang, and P.~Kokoszka.
\newblock Modeling probability density functions as data objects.
\newblock \emph{Econometrics and Statistics}, 21:\penalty0 159--178, 2022.

\bibitem[Pistone(2013)]{Pistone2013}
G.~Pistone.
\newblock Nonparametric information geometry.
\newblock In \emph{Lecture Notes in Computer Science}, pages 5--36. Springer
  Berlin Heidelberg, 2013.

\bibitem[Pistone and Sempi(1995)]{Pistone1995}
G.~Pistone and C.~Sempi.
\newblock An infinite-dimensional geometric structure on the space of all the
  probability measures equivalent to a given one.
\newblock \emph{The Annals of Statistics}, 23\penalty0 (5):\penalty0
  1543--1561, 1995.

\bibitem[Riabiz et~al.(2022)Riabiz, Chen, Cockayne, Swietach, Niederer, Mackey,
  and Oates]{Riabiz2022}
M.~Riabiz, W.~Y. Chen, J.~Cockayne, P.~Swietach, S.~A. Niederer, L.~Mackey, and
  C.~J. Oates.
\newblock Optimal thinning of {MCMC} output.
\newblock \emph{Journal of the Royal Statistical Society Series B: Statistical
  Methodology}, 84\penalty0 (4):\penalty0 1059--1081, 2022.

\bibitem[Schneider(2016)]{Schneider2016}
M.~Schneider.
\newblock Probability inequalities for kernel embeddings in sampling without
  replacement.
\newblock In \emph{Proceedings of the 19th International Conference on
  Artificial Intelligence and Statistics}, volume~51, pages 66--74, 2016.

\bibitem[Sch\"{o}lkopf and Smola(2002)]{Schlkopf2002}
B.~Sch\"{o}lkopf and A.~J. Smola.
\newblock \emph{Learning with Kernels}.
\newblock The {MIT} Press, 2002.

\bibitem[Sch\"{o}lkopf et~al.(1997)Sch\"{o}lkopf, Smola, and
  M\"{u}ller]{Schlkopf1997}
B.~Sch\"{o}lkopf, A.~Smola, and K.-R. M\"{u}ller.
\newblock Kernel principal component analysis.
\newblock In \emph{Lecture Notes in Computer Science}, pages 583--588. Springer
  Berlin Heidelberg, 1997.

\bibitem[Scholkopf et~al.(1999)Scholkopf, Mika, Burges, Knirsch, Muller,
  Ratsch, and Smola]{Scholkopf1999RS}
B.~Scholkopf, S.~Mika, C.~Burges, P.~Knirsch, K.-R. Muller, G.~Ratsch, and
  A.~Smola.
\newblock Input space versus feature space in kernel-based methods.
\newblock \emph{IEEE Transactions on Neural Networks}, 10\penalty0
  (5):\penalty0 1000--1017, 1999.

\bibitem[Scott et~al.(2016)Scott, Blocker, Bonassi, Chipman, George, and
  McCulloch]{Scott2016}
S.~L. Scott, A.~W. Blocker, F.~V. Bonassi, H.~A. Chipman, E.~I. George, and
  R.~E. McCulloch.
\newblock Bayes and big data: {T}he consensus {Monte Carlo} algorithm.
\newblock \emph{International Journal of Management Science and Engineering
  Management}, 11\penalty0 (2):\penalty0 78--88, 2016.

\bibitem[Sprungk(2020)]{Sprungk2020}
B.~Sprungk.
\newblock On the local lipschitz stability of {Bayesian} inverse problems.
\newblock \emph{Inverse Problems}, 36\penalty0 (5), 2020.

\bibitem[Sriperumbudur et~al.(2017)Sriperumbudur, Fukumizu, Gretton,
  Hyv\"{a}rinen, and Kumar]{Sriperumbudur2011inf}
B.~Sriperumbudur, K.~Fukumizu, A.~Gretton, A.~Hyv\"{a}rinen, and R.~Kumar.
\newblock Density estimation in infinite dimensional exponential families.
\newblock \emph{Journal of Machine Learning Research}, 18\penalty0
  (57):\penalty0 1--59, 2017.

\bibitem[Sriperumbudur et~al.(2010)Sriperumbudur, Gretton, Fukumizu,
  Sch\"{o}lkopf, and Lanckriet]{Sriperumbudur2010}
B.~K. Sriperumbudur, A.~Gretton, K.~Fukumizu, B.~Sch\"{o}lkopf, and G.~R.
  Lanckriet.
\newblock Hilbert space embeddings and metrics on probability measures.
\newblock \emph{Journal of Machine Learning Research}, 11:\penalty0 1517--1561,
  2010.

\bibitem[Srivastava et~al.(2015)Srivastava, Cevher, Dinh, and
  Dunson]{Srivastava2015}
S.~Srivastava, V.~Cevher, Q.~Dinh, and D.~Dunson.
\newblock {WASP: Scalable Bayes via barycenters of subset posteriors}.
\newblock In \emph{Proceedings of the Eighteenth International Conference on
  Artificial Intelligence and Statistics}, volume~38, pages 912--920, 2015.

\bibitem[Stuart(2010)]{Stuart2010}
A.~M. Stuart.
\newblock Inverse problems: A {Bayesian} perspective.
\newblock \emph{Acta Numerica}, 19:\penalty0 451–559, 2010.

\bibitem[Teymur et~al.(2021)Teymur, Gorham, Riabiz, and Oates]{Teymur2021}
O.~Teymur, J.~Gorham, M.~Riabiz, and C.~Oates.
\newblock Optimal quantisation of probability measures using maximum mean
  discrepancy.
\newblock In \emph{Proceedings of The 24th International Conference on
  Artificial Intelligence and Statistics}, volume 130, pages 1027--1035, 2021.

\bibitem[van~den Boogaart and Tolosana-Delgado(2013)]{vandenBoogaart2013}
K.~G. van~den Boogaart and R.~Tolosana-Delgado.
\newblock \emph{Analyzing Compositional Data with R}.
\newblock Springer Berlin Heidelberg, 2013.

\bibitem[van~den Boogaart et~al.(2014)van~den Boogaart, Egozcue, and
  Pawlowsky-Glahn]{Boogaart2014}
K.~G. van~den Boogaart, J.~J. Egozcue, and V.~Pawlowsky-Glahn.
\newblock Bayes {H}ilbert spaces.
\newblock \emph{Australian \& New Zealand Journal of Statistics}, 56\penalty0
  (2):\penalty0 171--194, 2014.

\bibitem[van~den Boogart et~al.(2011)van~den Boogart, Juan~José, and
  Vera]{Boogaart2010}
K.-G. van~den Boogart, E.~Juan~José, and P.-G. Vera.
\newblock Bayes linear spaces.
\newblock \emph{SORT-Statistics and Operational Research Transasctions},
  34\penalty0 (2):\penalty0 201--222, 2011.

\bibitem[Vyner et~al.(2023)Vyner, Nemeth, and Sherlock]{Vyner2023}
C.~Vyner, C.~Nemeth, and C.~Sherlock.
\newblock {SwISS}: A scalable markov chain monte carlo divide-and-conquer
  strategy.
\newblock \emph{Stat}, 12\penalty0 (1), 2023.

\bibitem[Wynne and Nagy(2021)]{Wynne2021}
G.~Wynne and S.~Nagy.
\newblock Statistical depth meets machine learning: Kernel mean embeddings and
  depth in functional data analysis.
\newblock \emph{arXiv:2105.12778}, 2021.

\bibitem[Xu et~al.(2022)Xu, Korba, and Slepcev]{Xu2022}
L.~Xu, A.~Korba, and D.~Slepcev.
\newblock Accurate quantization of measures via interacting particle-based
  optimization.
\newblock In \emph{Proceedings of the 39th International Conference on Machine
  Learning}, volume 162, pages 24576--24595, 2022.

\bibitem[Zhang et~al.(2021)Zhang, Khanna, Kyrillidis, and Koyejo]{Zhang2021}
J.~Zhang, R.~Khanna, A.~Kyrillidis, and S.~Koyejo.
\newblock Bayesian coresets: Revisiting the nonconvex optimization perspective.
\newblock In \emph{Proceedings of The 24th International Conference on
  Artificial Intelligence and Statistics}, volume 130, pages 2782--2790, 2021.

\end{thebibliography}

}

%%%%%%%%%%%%%%%%%%%%%%%%%%%%%%%%%%%%%%%%%%%%%%%%%%%%%%%%%%%%%%%%%%%%%%%%%%%%%%%%%%%%%%

%%%%%%%%%%%%%%%%%%%%%%%%%%%%%%%%%%%%%%%%%%%%%%%%%%%%%%%%%%%%%%%%%%%%%%%%%%%%%%%%%%%%%%

\end{document}